\PassOptionsToPackage{table}{xcolor}
\documentclass[final,hidelinks,onefignum,onetabnum]{siamart250211}

\usepackage{lipsum}
\usepackage{amsfonts}
\usepackage{graphicx}
\usepackage{epstopdf}
\usepackage{algorithmic}

\usepackage[T1]{fontenc} 
\ifpdf
  \DeclareGraphicsExtensions{.eps,.pdf,.png,.jpg}
\else
  \DeclareGraphicsExtensions{.eps}
\fi

\newsiamremark{remark}{Remark}
\newsiamremark{hypothesis}{Hypothesis}
\crefname{hypothesis}{Hypothesis}{Hypotheses}
\newsiamthm{claim}{Claim}
\newsiamremark{fact}{Fact}
\crefname{fact}{Fact}{Facts}

\headers{Optimization over the intersection of manifolds}{Y. Yang, B. Gao, and Y.-x. Yuan}

\title{Optimization over the intersection of manifolds\thanks{Submitted to the editors DATE.
\funding{This work was supported by the National Key R\&D Program of China (grant 2023YFA1009300). BG and YY were supported by the National Natural Science Foundation of China (grant No.~12288201).}}}

\author{Yan Yang\thanks{State Key Laboratory of Mathematical Sciences,
 Academy of Mathematics and Systems Science, Chinese Academy of Sciences,
 and University of Chinese Academy of Sciences, China
 (\email{yangyan@amss.ac.cn}).}
\and Bin Gao\thanks{State Key Laboratory of Mathematical Sciences,
 Academy of Mathematics and Systems Science, Chinese Academy of Sciences, China
 (\email{gaobin@lsec.cc.ac.cn}, \email{yyx@lsec.cc.ac.cn}).}
\and Ya-xiang Yuan\footnotemark[3]}

\usepackage{amsopn}
\DeclareMathOperator{\diag}{diag}

\ifpdf
\hypersetup{
  pdftitle={Optimization over the intersection of manifolds},
  pdfauthor={Yan Yang, Bin Gao, and Ya-xiang Yuan}
}
\fi

\usepackage{multirow} 
\usepackage{amsmath,amsxtra,amsfonts,amscd,amssymb,bm,bbm,mathrsfs}
\usepackage{nicefrac}       
\usepackage{extpfeil}
\usepackage{mathtools}
\usepackage{booktabs}
\usepackage{arydshln}
\usepackage[table]{xcolor} 
\usepackage{colortbl}
\usepackage{float}
\usepackage{hyperref}
\hypersetup{colorlinks=true,linkcolor=blue,citecolor=blue}
\usepackage{makecell}
\usepackage{lipsum}
\usepackage{graphicx}
\usepackage{epstopdf}
\usepackage{algorithmic}
\algsetup{indent=0.5em}
\usepackage{tcolorbox}

\newtheorem{assumption}{Assumption}
\crefname{assumption}{Assumption}{Assumptions}
\Crefname{assumption}{Assumption}{Assumptions}

\usepackage{booktabs}
\usepackage{arydshln}
\usepackage{multirow}

\usepackage{tikz}
\usetikzlibrary{positioning, arrows.meta,calc,decorations.text,bending}
\usetikzlibrary{arrows.meta, cd}
\usetikzlibrary{matrix,shadows,positioning,decorations.pathreplacing,decorations.markings}
\usetikzlibrary{matrix,shapes.arrows, shapes.geometric}
\usetikzlibrary{patterns,intersections}
\usepackage{pgfplots}
\usepgfplotslibrary{fillbetween}
\pgfplotsset{compat=1.18}
\usepackage{standalone}
\graphicspath{{./fig/}}

\usepackage{bm}
\usepackage{enumitem}

\newcommand{\innerp}[1]{\langle {#1} \rangle }
\newcommand{\norm}[1]{\left\| {#1} \right\| }

\newcommand{\hkh}[1]{\left\{ {#1} \right\}}
\newcommand{\kh}[1]{\left( {#1} \right)}
\newcommand{\expnumber}[2]{{#1}\mathrm{e}{#2}}

\DeclareMathOperator{\ddiag}{diag}
\DeclareMathOperator{\Diag}{Diag}
\DeclareMathOperator{\dist}{dist}
\DeclareMathOperator{\rank}{rank}

\DeclareMathOperator{\ima}{im}
\DeclareMathOperator{\dime}{dim}

\newcommand{\projection}{\mathcal{P}}
\newcommand{\manifold}{\mathcal{M}}

\newcommand{\aanifold}{\mathcal{A}}
\newcommand{\lanifold}{\mathcal{L}}
\newcommand{\wanifold}{\mathcal{W}}
\newcommand{\yanifold}{\mathcal{Y}}

\newcommand{\canifold}{\mathcal{C}}
\newcommand{\uanifold}{\mathcal{U}}
\newcommand{\vanifold}{\mathcal{V}}

\newcommand{\eanifold}{\mathcal{E}}

\newcommand{\kanifold}{\mathcal{K}}
\newcommand{\xanifold}{\mathcal{X}}

\newcommand{\sanifold}{\mathcal{S}}
\newcommand{\hanifold}{\mathcal{H}}
\newcommand{\hanifoldX}{{\hanifold_X}}
\newcommand{\Gh}{G_h}
\newcommand{\Gf}{G_f}

\newcommand{\matrixhyperboloid}{\mathbb{H}^m_{n}}

\newcommand{\lowrank}{\mathcal{M}_{s}}
\newcommand{\fixedrank}{\mathcal{M}_{r}}
\newcommand{\fixedrankplus}{\mathcal{M}_{r+1}}

\newcommand{\sparseset}{\canifold_s}

\newcommand{\stiefel}{\mathrm{St}}

\newcommand{\trace}{\mathrm{tr}}

\newcommand{\oblique}{\mathrm{Ob}}

\newcommand{\neighbor}{\mathcal{B}}

\newcommand{\diff}{\mathrm{D}}

\newcommand{\tangent}{\mathrm{T}}
\newcommand{\normal}{\mathrm{N}}
\newcommand{\retrac}{\mathrm{R}}
\newcommand{\tanL}{\mathrm{\bf L}}
\newcommand{\tanQ}{\mathrm{\bf Q}}

\newcommand{\Lnormal}{\mathrm{N}}

\newcommand{\Btangent}{\mathrm{T}}
\newcommand{\tangenttwo}{\mathrm{T}^2}

\newcommand{\grad}{\mathrm{grad}}

\newcommand{\tensorspace}{\mathbb{R}^{n_1\times n_2\times\cdots\times n_d}}
\newcommand{\vecr}{\mathbf{r}}

\newcommand{\mtt}{{\mathrm{tt}}}

\newcommand{\tensX}{\mathbf{X}}

\newcommand{\complexity}{O}
\newcommand{\frob}{\mathrm{F}}

\newcommand{\sym}{\mathbb{S}}

\newcommand{\mbR}{\mathbb{R}}

\newcommand{\mbRmn}{\mathbb{R}^{m\times n}}
\newcommand{\mbRm}{\mathbb{R}^{m}}

\definecolor{comblue}{RGB}{2,0,255}

\definecolor{blue}{rgb}{0, 0.451, 1}
\definecolor{myblue}{rgb}{0,0,.5}
\definecolor{bblue}{rgb}{0,0,.85}
\definecolor{mygreen}{rgb}{0, 0.62, 0.38}
\definecolor{red}{rgb}{0.796, 0.059, 0.063}
\definecolor{myred}{rgb}{.5,0,0}
\definecolor{LightGray}{rgb}{0.98, 0.98, 0.98}
\definecolor{Gray}{gray}{0.85}

\begin{document}

\maketitle

\begin{abstract}
Optimization over the intersection of two manifolds arises in a broad range of applications, but is hindered by the coupled geometry of the feasible region. In this paper, we prove that the regularities---clean intersection and intrinsic transversality---are equivalent, which yields a tractable {projection onto the} tangent space of the intersection. Therefore, we propose a geometric method that employs a retraction on only one manifold and updates the iterate along two orthogonal directions. Specifically, the iterates stay on one manifold, and the two directions are responsible for asymptotically approaching the other manifold and decreasing the objective function, respectively. Under intrinsic transversality, we derive the convergence rate for both the feasibility and optimality measures, and show that every accumulation point is first-order stationary. Numerical experiments on problems stemming from sparse and low-rank optimization, including fitting spherical data, approximating hyperbolic embeddings on real data, and computing compressed modes, demonstrate the effectiveness of the proposed method.
\end{abstract}



\section{Introduction}
In this work, we consider the following optimization problem over the intersection of two constraint sets:
\begin{equation}\label{eq:HM_opt}\tag{P}
    \begin{array}{cl}
    \displaystyle\min_{X\in\eanifold} & f(X) \\[1mm]
    \mathrm{s.\,t.} & h(X)=0, \\[1mm]
        & X\in\manifold,
    \end{array}
\end{equation}
where $\eanifold$ denotes a finite-dimensional Euclidean space, accommodating $\mbR^m$, $\mbR^{m\times n}$, or $\mbR^{n_1\times\cdots\times n_d}$, and $\manifold\subseteq\eanifold$ is a smooth submanifold. The objective function $f:\eanifold\to\mbR$ and the constraint-defining map $h:\eanifold\to\mbR^q$ are both smooth. We denote the zero level set of $h$ by 
\begin{equation*}
\hanifold := \hkh{X \in \eanifold \mid h(X) = 0},
\end{equation*}
and thus the feasible region is the intersection $\hanifold\cap\manifold$. Throughout this paper, we impose the following assumption.
\begin{assumption}\label{assu:h}
    There exists an open neighborhood $\kanifold\subseteq\eanifold$ of $\hanifold$ such that the differential $\diff h_X:\eanifold\to\mathbb{R}^q$ has full rank $q$ for all $X\in\kanifold$.
\end{assumption}
Assumption~\ref{assu:h} implies that $\hanifold$ is a smooth manifold in $\eanifold$; see~\cite[Corollary 5.14]{lee2012manifolds}.

In the vanilla scenario $h(\cdot)\equiv0$, i.e., $\hanifold=\eanifold$, \eqref{eq:HM_opt} reduces to an unconstrained optimization problem on the smooth manifold $\manifold$, for which a variety of algorithms---including Riemannian gradient descent and trust-region methods---are well established; see~\cite{absil2008optimization,boumal2023introduction}. However, once the additional constraint $X\in\hanifold$ is non-trivial, dealing with the intersection $\hanifold\cap\manifold$ presents several challenges. First, the intersection $\hanifold\cap\manifold$ does not necessarily constitute a smooth manifold, which impedes the direct application of existing Riemannian optimization algorithms. Second, the geometry of $\hanifold\cap\manifold$ is more intricate than that of $\manifold$ or $\hanifold$ alone---for instance, unclear characterization of the tangent cone to the intersection hinders the construction of effective search directions. Third, projections onto $\hanifold\cap\manifold$ are generally unavailable in closed form, making it difficult to preserve the feasibility of the iterates.

\subsection{Motivation and applications}\label{sec:application}
The formulation~\eqref{eq:HM_opt} encompasses a broad range of problems where a manifold constraint $X\in\manifold$ is coupled with additional structured requirements $h(X)=0$. We outline several representative applications.

\begin{table}[htbp]
    \setlength{\extrarowheight}{1ex}
    \centering
    \caption{Instances of manifold intersections arising in sparse optimization, low-rank matrix and tensor optimization. The intersection geometry of the specific $(\manifold,\hanifold)$ pairs is summarized. Specifically, all the intersections satisfy the intrinsic transversality.}
    \label{tab:tangentsets}
    \vspace{0.1cm}
    \small
        \begin{tabular*}{\textwidth}{@{\extracolsep{\fill}}llll}
            \toprule
             & {Manifold} & {Level set of $h$} & {Intersection geometry} \\
            \midrule
            \multirow{2}{*}{Sparse}
            & $\sparseset$ & $\{X\in\mbR^m\mid \|X\|^2_\frob=1\}$ & \cite{beck2016minimizationC_B} \\
            & $\sparseset$ & $\{X\in\mbR^{n\times p}\mid X^\top X=I_p\}$ & \cite{chenhuang2026sparse_stiefel} \\
            \midrule
            \multirow{7}{*}{Low-rank}
            & $\fixedrank$ & $\{X\in\mbRmn\mid\aanifold(X)=b\}$ & \cite{li2023normalboundedaffine,yang20252ndVariety} \\
            & $\fixedrank$ & $\hanifold$ is orthogonally invariant & \cite{yang2025spacedecouple,yang20252ndVariety} \\
            & $\fixedrank$ & $\hanifold$ is hyperbolic~\eqref{eq:matrixhyperboloid} & Appendix~\ref{app:proj-hyperboloid} \\
            & $\sym_{r}(n)$ & $\{X\in\sym(n)\mid\|X\|^2_\frob=1\}$ & \cite{cason2013iterative,li2020jotaspectral,yang20252ndVariety} \\
            & $\sym^+_r(n)$ & $\{X\in\sym(n)\mid\aanifold(X)=b\}$ & \cite{levin2025effect,yang20252ndVariety} \\
            & $\manifold^{\mtt}_{\vecr}$ & $\{\tensX\in\tensorspace\mid \|\tensX\|_\frob^2=1\}$ & \cite{peng2025normalized,yang20252ndVariety} \\
            \midrule
            General & $\manifold$ & $\hanifold$ satisfies Assumption~\ref{assu:h} & Intrinsic transversality \\
            \bottomrule
        \end{tabular*}
\end{table}

Let $\|\cdot\|_0$ denote the cardinality of an element. The sparsity set $\manifold=\sparseset:=\{X\in\eanifold\mid\|X\|_0=s\}$ combined with normalization or orthogonality constraints appears in several scenarios. When $\eanifold=\mbRm$ is the vector space, Beck and Hallak~\cite{beck2016minimizationC_B} characterized the projection onto $\sparseset\cap\{X\in\mbRm\mid\|X\|_\frob^2=1\}$, with applications in the fields of genetics and finance. Moreover, when $\eanifold=\mbR^{n\times p}$, the intersection of $\sparseset$ and the Stiefel manifold $\stiefel(n,p):=\{X\in\mbR^{n\times p}\mid X^\top X=I_p\}$ underlies sparse principal component analysis~\cite{chen2020ManPG} and the geometry of $\stiefel(n,p)\cap\canifold_s$ has been recently studied in~\cite{chenhuang2026sparse_stiefel}.

Another important class of instances concerns the fixed-rank manifold $\manifold=\fixedrank:=\{X\in\mbRmn\mid \rank(X)=r\}$. Specifically, Cason et al.~\cite{cason2013iterative} derived the tangent cone to $\hanifold\cap\fixedrank$ with $\hanifold$ as the Frobenius sphere, applied to the approximation of graph similarity matrices. Li and Luo~\cite{li2023normalboundedaffine} subsequently obtained the normal cone to $\hanifold\cap\fixedrank$ for $\hanifold$ as an affine manifold. More generally, Yang et al.~\cite{yang2025spacedecouple} characterized the tangent and normal cones to $\hanifold\cap\fixedrank$ when $h$ is \emph{orthogonally invariant}, i.e., $h(X)=h(XQ)$ for all orthogonal $Q$. In addition, the hyperbolic constraint, arising from low-rank compression of hyperbolic embeddings for hierarchical data~\cite{jawanpuria2019lowrankhyperbolic}, was treated in~\cite{yang20252ndVariety}. In the symmetric setting, Li et al.~\cite{li2020jotaspectral} studied the geometry of $\sym_{r}(n):=\{X\in\sym(n)\mid \rank(X)=r\}$ intersected with spectral constraints, where $\sym(n):=\{X\in\mbR^{n\times n}\mid  X^\top=X\}$. Levin et al.~\cite{levin2025effect} analyzed the positive semidefinite counterparts, denoted by $\hanifold\cap\sym^+_{r}(n)$ for some $\hanifold$. The results were extended to tensors: the set $\manifold^{\mtt}_{\vecr}$ of low-rank tensors in tensor-train format coupled with a sphere was investigated in~\cite{peng2025normalized} with applications to quantum physics. A unified analysis on the geometry of $\hanifold\cap\manifold$ covering the above low-rank instances was developed in~\cite{yang20252ndVariety}. We refer the reader to Table~\ref{tab:tangentsets} for a summary.

\vspace{-0.5mm}
\subsection{Related work}\label{sec:relatedwork}
We begin with $\manifold=\eanifold$ and the role of decomposing search directions into orthogonal components for tackling the equality constraint $h(X)=0$.

\textbf{Direction decomposition for equality-constrained problems.} If the region $\manifold$ is the whole Euclidean space $\eanifold$, \eqref{eq:HM_opt} reduces to an optimization problem solely subject to the equality constraint $h(X)=0$. A representative principle is to decompose the update direction into tangent and normal components, responsible for decreasing the objective $f$ and for encouraging the feasibility $X\in\hanifold$, respectively. This idea was first formalized by \cite{rosen1960gradientI,rosen1961gradientII}, and then Frost~\cite{frost1972CGP} proposed the \emph{corrective gradient projection} method, which realized a normal direction as the correction to the tangent one. Additionally, second-order information was exploited in the so-called \emph{null-space methods}~\cite{nocedal1985projectedHessian,yuan2001nullspace}, which enhanced the convergence results. More relevant to our work, the \emph{landing algorithm} was proposed for optimization problems with orthogonality constraints~\cite{ablin2022fastlanding,gao2022landingStiefel,ablin2024infeasible}, getting rid of computationally expensive retractions adopted by Riemannian optimization methods~\cite{absil2008optimization,boumal2023introduction}. Then, Schechtman et al.~\cite{schechtman2023ODCGM} extended the technique to general equality constraints, accommodating stochastic oracles, and Vary et al.~\cite{vary24landinggeneralStiefel} adapted the algorithm for problems over the random generalized Stiefel manifold, which was further extended to distributed optimization~\cite{song2025distributedlanding}. Subsequent work~\cite{si2026unifiedlanding,goyens2026riemannian} incorporated backtracking line search into the landing method. More recently, Xiong et al.~\cite{xiong2026langding2} designed a second-order landing algorithm to achieve locally quadratic convergence. Goyens and Feppon~\cite{goyens2026riemannian} unveiled the relationship between the landing algorithm and several classical optimization methods including  the \emph{sequential quadratic programming method} and the \emph{augmented Lagrangian method}~\cite{nocedal2006numerical}.

\smallskip

When the manifold constraint $X\in\manifold$ is non-trivial, i.e., $\manifold\neq\eanifold$, finding feasible points in $\hanifold\cap\manifold$ is itself a challenging problem, and thus we review the theory of the method of alternating projections (MAP)~\cite{neumann1950functional}. 

\textbf{Intersection condition and alternating projection.} Given two general closed sets $\hanifold$ and $\manifold$ in a Euclidean space, the \emph{feasibility problem} seeks a point $X^*\in\hanifold\cap\manifold$. The method of alternating projections generates iterates by $X_{k+1} \in\projection_{\manifold}(\projection_{\hanifold}(X_k))$, where $\projection$ denotes the projector onto a closed set.

When $\hanifold$ and $\manifold$ are nonconvex, establishing local linear convergence requires appropriate regularity conditions of the intersection $\hanifold\cap\manifold$, often realized as a separation property of the limiting normal cones. The conditions proposed in existing work are summarized in Figure~\ref{fig:intersection_conditions}. Specifically, Lewis and Malick~\cite{lewis2008alternatingmanifolds} first established local linear convergence of the MAP under \emph{transversality} when both $\hanifold$ and $\manifold$ are manifolds, which was generalized to \emph{clean intersection} by Andersson and Carlsson~\cite{andersson2013alternating_cleanintersection}. In parallel, Lewis et al.~\cite{lewis2009localaveraged} introduced \emph{linear regularity} for general closed sets, and Bauschke et al.~\cite{bauschke2013restrictedtheory,bauschke2013restrictedapplication} weakened it to the \emph{restricted regularity}. The regularity conditions were further weakened to \emph{intrinsic transversality}~\cite{drusvyatskiy2015intrinsictransversality} and the \emph{separable condition}~\cite{noll2016separableMAP}, respectively. In addition, the local convergence can be preserved when the exact projections are replaced with appropriate inexact ones~\cite{drusvyatskiy2019inexactapproximate,budzinskiy2025quasioptimal,xiao2025quadraticMAP}. The recent work~\cite{chen2026retractionsalternatingprojections} reveals that the alternating projections can further induce retractions over manifold intersections. To the best of our knowledge, intrinsic transversality appears to be one of the most general conditions to guarantee local linear convergence of the MAP. 

\begin{figure}[htbp]
\centering
\resizebox{\textwidth}{!}{%
\begin{tikzpicture}[
    >=stealth,
    box/.style={rectangle, rounded corners=4pt, draw=black, thick, dashed,
        fill=gray!3, minimum height=0.8cm,
        align=center, font=\small, inner sep=7pt},
    arr/.style={->, thick, draw=black!50, shorten >=3pt, shorten <=3pt}
]
\node[box, minimum width=3.5cm] (trans) at (0,0)
    {Transversality \cite{lewis2008alternatingmanifolds}};
\node[box, minimum width=4.0cm] (clean) at (5.0,0)
    {Clean intersection \cite{andersson2013alternating_cleanintersection}};
\node[box, minimum width=4.8cm] (IT) at (10.6,0)
    {Intrinsic transversality \cite{drusvyatskiy2015intrinsictransversality}};

\node[box, minimum width=3.5cm] (linreg) at (0,-1.5)
    {Linear regularity \cite{lewis2009localaveraged}};
\node[box, minimum width=4.0cm] (restreg) at (5.0,-1.5)
    {Restricted regularity \cite{bauschke2013restrictedtheory}};
\node[box, minimum width=4.8cm] (sep) at (10.6,-1.5)
    {Separable condition \cite{noll2016separableMAP}};

\draw[arr] (trans.east) -- (clean.west);
\draw[arr] (clean.east) -- (IT.west);
\draw[arr] (linreg.east) -- (restreg.west);
\draw[arr] (restreg.east) -- (sep.west);
\draw[arr] (trans.south) -- (linreg.north);
\draw[arr] (clean.south) -- (restreg.north);
\draw[arr] ([yshift=10pt]restreg.east) -- ([yshift=-10pt]IT.west);
\draw[arr] ([yshift=-10pt]clean.east) -- ([yshift=10pt]sep.west);
\end{tikzpicture}%
}
\vspace{-4mm}
\caption{Development of intersection regularity conditions for local linear convergence of the MAP. An arrow from $A$ to $B$ indicates that $A$ implies $B$.}
\label{fig:intersection_conditions}
\end{figure}
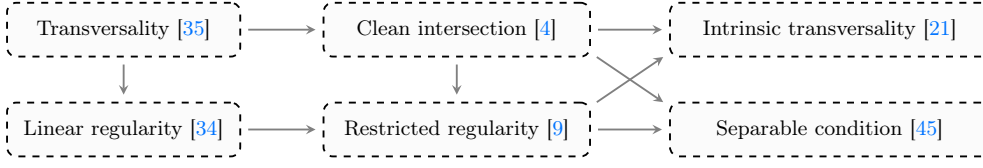

\smallskip

To address the optimization problem~\eqref{eq:HM_opt}, existing methods harness the specific structure of the constraint sets; we summarize them as follows.

\textbf{Optimization over the intersection of sets.} For a closed and convex $\manifold\subseteq\eanifold$, a variety of methods have been developed for~\eqref{eq:HM_opt}; see~\cite{rockafellar1976ALMconvex,andreani2008ALM_lowerlevel,xiao2025exactconvex}. When a smooth manifold $\manifold$ is considered (generally nonconvex), a common perspective in recent literature is to cast~\eqref{eq:HM_opt} as a Riemannian optimization problem on $\manifold$ with nonlinear constraints $h(X)=0$. The constraint qualifications and optimality conditions have been investigated in~\cite{yang2014optimalityRie_nonlinear,bergmann2019intrinsicKKT,andreani2024CQ_RieALM,andreani2026globalRALM}. Several Riemannian augmented Lagrangian methods were proposed~\cite{liu2020simple,zhou2023semismoothNewton,jia2023augmented,andreani2024CQ_RieALM,andreani2026globalRALM}, which handle the constraints $h(X) = 0$ by addressing a sequence of surrogate subproblems in the form of $\min_{X\in \manifold}\ f(X)+\langle \lambda, h(X)\rangle + \frac{\mu}{2}\|h(X)\|^2$. Moreover, second-order methods have also been developed for~\eqref{eq:HM_opt}, including Riemannian sequential quadratic optimization methods~\cite{schiela2021Riesqp,obara2022RieSQP} and Riemannian interior point methods \cite{lai2024riemannianint}. However, all the aforementioned Riemannian methods require solving a subproblem at each iteration.

\subsection{Contributions} In this work, we aim to develop a geometric method for~\eqref{eq:HM_opt} by exploiting the structure of the manifold intersection. Under Assumption~\ref{assu:h}, the differential of $h$ has full rank in the neighborhood $\kanifold$, and thus the set 
\[
\hanifoldX:=\{\tilde{X}\in\eanifold\mid h(\tilde{X})=h(X)\},
\]
serving as a perturbation of $\hanifold$, is a smooth manifold for all $X\in\kanifold$. In addition, we assume the intrinsic transversality condition (see Definition~\ref{def:intrinsic_transversality}) as follows.

\begin{assumption}\label{assu:IT}
For all $X\in\kanifold\cap\manifold$, the manifolds $\hanifoldX$ and $\manifold$ are intrinsically transversal at $X\in\hanifoldX\cap\manifold$.
\end{assumption}

To alleviate the difficulty arising from the intricate coupling of $\hanifold$ and $\manifold$ in the constraints of~\eqref{eq:HM_opt}, we preserve $X\in\manifold$ along the iterates via the retraction on $\manifold$, in the spirit of the Riemannian optimization framework. More importantly, we interpret~\eqref{eq:HM_opt} as two sub-tasks: identifying feasible points in $\hanifold\cap\manifold$ and decreasing the objective $f$. This perspective, together with the decomposition principle introduced in section~\ref{sec:relatedwork}, inspires us to seek two orthogonal directions in the tangent space of $\manifold$ that handle the two tasks respectively: a \emph{feasibility direction $\Gh$} that drives the iterates towards $\hanifold\cap\manifold$, and an \emph{optimality direction $\Gf$} that accounts for the descent of $f$. The resulting update rule takes the following form,
\begin{equation*}
    X_{k+1}=\retrac_{X_k}^\manifold(\alpha_k \Gh(X_k)+\beta_k\Gf(X_k)),
\end{equation*}
where $\alpha_k$ and $\beta_k$ are the step sizes.
We then concentrate on constructing the two orthogonal directions tangent to $\manifold$, which resorts to the intersection geometry. Central to the development are two new equivalent characterizations of intrinsic transversality.

For the feasibility direction, we project the Gauss--Newton direction---an approximation of $(\projection_{\hanifold}(X)-X)$ that pushes $X$ toward $\hanifold$---onto the tangent space of $\manifold$. We prove in Theorem~\ref{thm:PB-implies-IT} that intrinsic transversality is equivalent to a \emph{projection-based transversality condition}; this equivalence ensures that the projected Gauss--Newton direction retains a sufficient tangential component, thereby providing an effective improvement on the feasibility. For the optimality direction, we project the negative gradient $-\nabla f(X)$ onto the tangent cone $\tangent_{\hanifoldX\cap\manifold}(X)$. To compute this projection, we establish in Theorem~\ref{thm:IT_implies_clean} that, given two manifolds generally, intrinsic transversality is equivalent to clean intersection, answering an open question posed in~\cite[\S8]{ioffe2017transversality}. Therefore, $\hanifoldX\cap\manifold$ is a manifold with the following intersection rule,
\begin{equation}\label{eq:tangent_decom}
\tangent_{\hanifoldX\cap\manifold}(X) = \tangent_{\hanifoldX}(X)\cap\tangent_{\manifold}(X),
\end{equation}
Consequently, the above identity reveals that $\tangent_{\hanifoldX\cap\manifold}(X)$ is a linear subspace, and the projection onto it is characterized explicitly in Proposition~\ref{pro:proj_SX}.

Combining the feasibility and optimality directions in the tangent space, we propose the \emph{Geometric method via Orthogonal Tangent Directions} (GOTD), with the iterates staying on $\manifold$, asymptotically approaching $\hanifold$, and decreasing $f$. The implementation is summarized in Algorithm~\ref{alg:GOTD} and is illustrated in Figure~\ref{fig:gotd_illustration}. Moreover, we prove an $\complexity(1/\sqrt{K})$ convergence rate for both the feasibility and the optimality measures in Theorem~\ref{thm:lyap_complexity}, and show that every accumulation point is first-order stationary for~\eqref{eq:HM_opt} in Corollary~\ref{cor:stationary}, under some constraint qualifications. 

Numerical experiments on fitting spherical data, approximating hyperbolic embeddings on real data, and computing compressed modes demonstrate the effectiveness and efficiency of GOTD, attributable to the exploitation of the intersection geometry.

\definecolor{cyanbox}{RGB}{0,160,180}
\definecolor{redbox}{RGB}{200,60,60}
\colorlet{feasicolor}{cyanbox!80!black}
\colorlet{opticolor}{redbox!80!black}
\colorlet{fig@Gn}{gray}
\colorlet{fig@Gt}{gray}
\colorlet{fig@GnLabel}{black}
\colorlet{fig@GtLabel}{black}
\colorlet{fig@curve}{feasicolor}
\colorlet{fig@star}{cyanbox!60!black}
\colorlet{fig@curveLabel}{feasicolor}

\begin{figure}[htbp]
\centering
\includegraphics[width=0.75\textwidth]{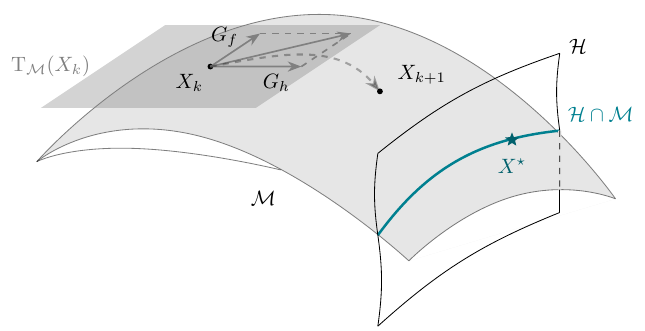}
\caption{Illustration of GOTD. At $X_k\in\manifold$, the update direction combines two orthogonal components in the tangent space $\tangent_\manifold(X_k)$: the feasibility direction $\Gh$ that drives $X_k$ towards $\hanifold$ and the optimality direction $\Gf$ that decreases $f$. The $X_{k+1}$ is obtained by a retraction on $\manifold$.}
\label{fig:gotd_illustration}
\vspace{-2mm}
\end{figure}

\subsection{Organization}
Section~\ref{sec:prelim} presents the notation and preliminaries. We propose the main framework in section~\ref{sec:framework}, and construct in sections~\ref{subsec:feasibility}--\ref{subsec:clean} the feasibility and optimality directions via two equivalent characterizations of intrinsic transversality. Section~\ref{sec:convergence} provides the convergence analysis. Section~\ref{sec:experiments} validates the method on sparse and low-rank optimization problems, and the conclusion is drawn in section~\ref{sec:conclusion}. 


\section{Notation and preliminaries}\label{sec:prelim}
We outline the notation adopted in this paper, and then review some background in variational analysis and Riemannian optimization; see \cite{absil2008optimization,lewis2008alternatingmanifolds,rockafellar2009variationalanalysis,lee2012manifolds,boumal2023introduction} for more details and references.

\subsection{Notation}
The diagonal matrix with entries $x$ is written as $\Diag(x)$, and $\ddiag(X)$ extracts the diagonal of $X$ as a vector. On a Euclidean space, we adopt the Frobenius inner product $\innerp{X_1,X_2}:=\trace(X_1^\top X_2)$, with the induced norm $\|X\|:=\sqrt{\innerp{X,X}}$. For a set $\xanifold\subseteq\eanifold$, the distance from $Y$ to $\xanifold$ is $\dist(Y,\xanifold):=\inf_{X\in\xanifold}\|X-Y\|$, and $\projection_\xanifold$ stands for the projection onto $\xanifold$. When $\xanifold$ is additionally a smooth manifold, $\tangent_\xanifold(X)$ refers to its tangent space at $X$, and any smooth map $F:\xanifold_1\to\xanifold_2$ between manifolds admits the differential $\diff F_X:\tangent_{\xanifold_1}(X)\to\tangent_{\xanifold_2}(F(X))$ at $X$. Given a matrix $X$ of rank $s$, we write its singular value decomposition by $X=U\varSigma V^\top$ with $U\in\stiefel(m,s)$, $\varSigma\in\mbR^{s\times s}$, and $V\in\stiefel(n,s)$; then the Moore--Penrose inverse is $X^\dagger=V\varSigma^{-1}U^\top$. Given a map $F:\eanifold\to\eanifold^\prime$ between two Euclidean spaces, we use $\ima(F)\subseteq\eanifold^\prime$ to denote the image and $F^*:\eanifold^\prime\to\eanifold$ to denote the adjoint operator. The operator $\odot$ denotes the Hadamard (entry-wise) product.

\subsection{Preliminaries}
Let $\xanifold$ be a \emph{locally closed} set in a finite-dimensional Euclidean space $\eanifold$, i.e., every point in $\xanifold$ admits a closed neighborhood $\neighbor\subseteq\eanifold$ such that $\neighbor\cap\xanifold$ is closed in $\eanifold$. The \emph{Bouligand tangent cone} to $\xanifold$ at a point $X\in\xanifold$ is
\begin{equation}\label{eq:tangentcone}
    \begin{aligned}
        \tangent_\xanifold(X) := \left\{ \eta\in\eanifold\mid \text{there exists}\ t_i\to0,\,\text{such that}\,\dist(X+t_i\eta,\xanifold)=o(t_i)\right\}.
    \end{aligned}
\end{equation}
Taking the polar operation on $\tangent_\xanifold(X)$ yields the \emph{Fr\'echet normal cone}, 
\begin{equation*}
\normal_\xanifold(X):=\kh{\tangent_\xanifold(X)}^\circ = \hkh{ Y\in\eanifold \mid \langle Y, \eta\rangle \le 0, \ \text{for all}\ \eta \in \Btangent_\xanifold(X)}.
\end{equation*}
When $\xanifold$ is a smooth manifold, the tangent and normal cones coincide with the tangent and normal spaces, respectively. We then consider the intersection of two sets $\xanifold$ and $\yanifold$. Given $X\in \xanifold\cap \yanifold$, it generally holds that
\begin{equation}\label{eq:cone_oneside}
    \Btangent_{\xanifold\cap \yanifold}(X) \subseteq  \Btangent_{\xanifold}(X) \cap \Btangent_{\yanifold}(X)\quad \text{and} \quad\ 
    \normal_{\xanifold\cap \yanifold}(X) \supseteq \normal_{\xanifold}(X) + \normal_{\yanifold}(X).
\end{equation}
If $\xanifold$ and $\yanifold$ are smooth manifolds and they intersect \emph{transversally} at $X\in\xanifold\cap\yanifold$, i.e., 
\begin{equation}\label{eq:transversally}
    \tangent_{\xanifold}(X) + \tangent_{\yanifold}(X) = \eanifold,\ \ \text{or equivalently},\ \ \normal_{\xanifold}(X)\cap \normal_{\yanifold}(X) = \{0\},
\end{equation}
then $\xanifold\cap \yanifold$ is also a smooth manifold around $X$ with
\begin{equation}\label{eq:cone_tranverse}
    \Btangent_{\xanifold\cap \yanifold}(X) = \Btangent_{\xanifold}(X) \cap \Btangent_{\yanifold}(X)\quad \text{and} \quad\ 
    \normal_{\xanifold\cap \yanifold}(X) = \normal_{\xanifold}(X) + \normal_{\yanifold}(X).
\end{equation}
Note that the definition of transversality~\eqref{eq:transversally} depends on the ambient space $\eanifold$, and thus a generalization called \emph{clean intersection}~\cite{Hormander1985III} was considered in~\cite{andersson2013alternating_cleanintersection} for the convergence analysis of alternating projections.
\begin{definition}[Clean intersection]\label{def:clean_intersection}
    Two manifolds $\xanifold, \yanifold \subseteq \eanifold$ \emph{intersect cleanly} at $Z\in\xanifold\cap\yanifold$ if $\xanifold\cap\yanifold$ is a smooth manifold in a neighborhood $\neighbor$ of $Z$ and it holds that $\tangent_{\xanifold\cap\yanifold}(X) = \tangent_{\xanifold}(X)\cap\tangent_{\yanifold}(X)$ for all $X\in\xanifold\cap\yanifold\cap\neighbor$.
\end{definition}
Moreover, Drusvyatskiy et al.~\cite{drusvyatskiy2015intrinsictransversality} introduced the following notion, which characterizes the intersection via pairs of nearby points.

\begin{definition}[Intrinsic transversality]\label{def:intrinsic_transversality}
    Two locally closed sets $\xanifold, \yanifold \subseteq \eanifold$ are \emph{intrinsically transversal} at $Z \in \xanifold \cap \yanifold$ if there exists a constant $\kappa \in (0,1]$ such that, for all $X \in \xanifold \setminus \yanifold$ and $Y \in \yanifold \setminus \xanifold$ in a neighborhood of $Z$,
    \begin{equation}\label{eq:IT_kappa}
        \max\hkh{ \dist(u, \normal_\yanifold(Y)), \dist(u, -\normal_\xanifold(X)) } \ge \kappa,
    \end{equation}
    where $u = (X-Y)/\|X-Y\|$. The $\kappa$ is called the constant of intrinsic transversality.
\end{definition}
Intrinsic transversality reveals that the difference direction $u$ cannot lie close to $\Lnormal_\yanifold(Y)$ and $-\Lnormal_\xanifold(X)$ simultaneously, reflecting a separation property of the two cones. When $\xanifold$ and $\yanifold$ are manifolds, clean intersection implies intrinsic transversality~\cite{drusvyatskiy2015intrinsictransversality}.

For problems constrained on a smooth manifold $\xanifold$, the framework of Riemannian optimization is developed, leveraging the Riemannian geometry of manifolds; see~\cite{absil2008optimization,boumal2023introduction} for an overview. To guide the movement from the current point along a tangent vector, a geometric tool \emph{retraction} is introduced. Specifically, a smooth map $\retrac^{\xanifold}: \tangent\xanifold \rightarrow \xanifold$, defined on the tangent bundle $\tangent \xanifold:=\bigcup_{X\in\xanifold}\tangent_{\xanifold}(X)$, is called a {retraction} on the manifold $\xanifold$ if for all $X\in\xanifold,\xi\in\tangent_{\xanifold}(X)$, the curve $\gamma(t):=\retrac^{\xanifold}_X(t\xi)$ satisfies $\gamma(0)=X$ and $\gamma^\prime(0)=\xi$, where $\retrac^{\xanifold}_X$ denotes the restriction of $\retrac^{\xanifold}$ on $\tangent_{\xanifold}(X)$.

\section{A geometric framework via orthogonal tangent directions}\label{sec:framework}
We now present a geometric framework for problem~\eqref{eq:HM_opt}. To preserve the structure of $\manifold$, we adopt the following retraction-based update rule:
\begin{equation}\label{eq:general_updateX}
    X_{k+1}=\retrac_{X_k}^\manifold(\alpha_k \Gh(X_k)+\beta_k\Gf(X_k)),
\end{equation}
where $\Gh(X_k),\Gf(X_k)\in\tangent_{\manifold}(X_k)$. Interpreting~\eqref{eq:HM_opt} as two sub-tasks---decreasing $\dist(\cdot,\hanifold)$ for feasibility and decreasing $f(\cdot)$ for optimality---we then design $\Gh$ and $\Gf$ to handle the tasks respectively, inspired by the decomposition principle introduced in section~\ref{sec:relatedwork}. Under Assumptions~\ref{assu:h}--\ref{assu:IT} and given an iterate $X\in\kanifold\cap\manifold$, we recall the level set $\hanifoldX=\{\tilde{X}\in\eanifold\mid h(\tilde{X})=h(X)\}$ and denote the tangent cone by $S(X):=\tangent_{\hanifoldX\cap\manifold}(X)$. The following observation motivates the construction.

\begin{lemma}\label{lem:orth_GtGn}
    Given two manifolds $\xanifold$ and $\yanifold$ in $\eanifold$ with a point $X\in\xanifold\cap\yanifold$. Then for all $d\in\normal_{\yanifold}(X)$ and $\eta\in\eanifold$, it holds that $\left\langle\projection_{\tangent_{\xanifold\vphantom{\cap\yanifold}}(X)}(d),\,\projection_{\tangent_{\xanifold\cap\yanifold}(X)}(\eta)\right\rangle = 0$.
\end{lemma}
\begin{proof}
    Let $\bar \eta = \projection_{\tangent_{\xanifold\cap\yanifold}(X)}(\eta)$. By~\eqref{eq:cone_oneside}, we have $\bar\eta\in\tangent_{\xanifold\cap\yanifold}(X)\subseteq\tangent_{\xanifold}(X)$, and thus $\innerp{\projection_{\tangent_{\xanifold}(X)}(d),\bar\eta} = \innerp{d,\bar\eta}$. Similarly, the inclusion $\bar\eta\in\tangent_{\yanifold}(X)$ holds, which together with $d\in\normal_{\yanifold}(X)$ yields $\innerp{d,\bar\eta}=0$.
\end{proof}

Applying Lemma~\ref{lem:orth_GtGn} with $(\xanifold,\yanifold)=(\manifold,\hanifoldX)$, the orthogonality $\innerp{\Gh,\Gf}=0$ is guaranteed whenever the two directions take the following form,
\begin{equation}\label{eq:constructionofGnGt}
    \Gh(X)=\projection_{\tangent_{\manifold}(X)}(d(X))\ \ \text{and}\ \ \Gf(X)=\projection_{S(X)}(\eta(X)),
\end{equation}
where $d(X)\in\normal_{\hanifoldX}(X)$ and $\eta(X)\in\eanifold$. Therefore, we shift our focus toward the design of $d(X)$ and $\eta(X)$ in sections~\ref{subsec:feasibility} and~\ref{subsec:clean}, respectively, which mainly rely on the geometry of the manifold intersection. The overall development is illustrated in Figure~\ref{fig:flowchart}, and the geometric framework is realized in Algorithm~\ref{alg:GOTD}.

\begin{figure}[htbp]
\centering
\begin{tikzpicture}[>=Stealth, scale=0.85, every node/.style={scale=0.85},
    bx/.style={rectangle, rounded corners=4pt, draw=black, thick, dashed,
        fill=gray!5, minimum height=0.9cm, align=center, inner sep=6pt},
    cbx/.style={rectangle, rounded corners=4pt, draw=feasicolor, thick, dashed,
        fill=feasicolor!5, minimum height=0.9cm, align=center, inner sep=6pt},
    rbx/.style={rectangle, rounded corners=4pt, draw=opticolor, thick, dashed,
        fill=opticolor!5, minimum height=0.9cm, align=center, inner sep=6pt},
    arr/.style={->, thick, line width=1pt}
]
    \def\vsep{1.}
    \node[bx, minimum width=2.5cm] (P) at (0,0) {
        Problem~\eqref{eq:HM_opt}
    };

    \node[cbx, minimum width=2.8cm] (F) at (3.3, \vsep) {
        {Feasibility}\\[1mm]
        {\small $\dist(\cdot,\hanifold)\downarrow$}
    };

    \node[rbx, minimum width=2.8cm] (O) at (3.3,-\vsep) {
        {Optimality}\\[1mm]
        {\small $f(\cdot)\downarrow$}
    };

    \node[cbx, minimum width=2.8cm] (Gn) at (8.8, \vsep) {
        {\small $\Gh(X)\in\tangent_{\manifold}(X)$}
    };
    \node[rbx, minimum width=2.8cm] (Gt) at (8.8,-\vsep) {
        {\small $\Gf(X)\in\tangent_{\manifold}(X)$}
    };

    \draw[<->, thin, black!60, shorten >=4pt, shorten <=4pt] (Gn.south) -- node[left, font=\normalsize] {orthogonal} (Gt.north);

    \node[bx, minimum width=2.8cm] (U) at (12.5,0) {
        {GOTD}\\[1mm]
        {\small Algorithm~\ref{alg:GOTD}}
    };

    \draw[arr, feasicolor!80!black] ($(P.east)+(0,0.3)$) -- (F.west);
    \draw[arr, opticolor!80!black] ($(P.east)+(0,-0.3)$) -- (O.west);

    \draw[arr, feasicolor!80!black] (F.east) -- node[above, font=\normalsize, black] {Theorem~\ref{thm:PB-implies-IT}} (Gn.west);
    \draw[arr, opticolor!80!black] (O.east) -- node[below, font=\normalsize, black] {Theorem~\ref{thm:IT_implies_clean}} (Gt.west);

    \draw[arr, feasicolor!80!black] (Gn.east) -- ($(U.west)+(0,0.3)$);
    \draw[arr, opticolor!80!black] (Gt.east) -- ($(U.west)+(0,-0.3)$);

\end{tikzpicture}
\caption{Development of the GOTD algorithm: two equivalent characterizations of intrinsic transversality guide the construction of the orthogonal tangent directions $\Gh$ and $\Gf$.}
\label{fig:flowchart}
\end{figure}
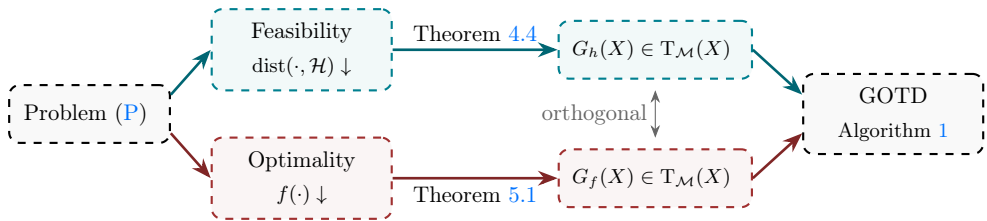

We outline some lemmas in preparation for sections~\ref{subsec:feasibility}--\ref{subsec:clean}; the reader is invited to refer back as appropriate.

\subsection{Auxiliary results on differential manifolds}\label{app:manifold-aux} The following lemma reveals that the difference between two nearby points is nearly tangent.

\begin{lemma}\label{lem:chord-tangent}
Let $\xanifold$ be a $C^2$ embedded submanifold of $\eanifold$ and let $Z\in\xanifold$. Then there exist a neighborhood $\neighbor$ of $Z$ and a constant $C_\xanifold>0$ such that for all $X_1, X_2\in\xanifold\cap\neighbor$,
\[
    \dist\big(X_1-X_2,\,\tangent_{\xanifold}(X_2)\big)
    \le C_{\xanifold}\|X_1-X_2\|^2.
\]
\end{lemma}
\begin{proof}
By the $C^2$ regularity of $\xanifold$, there exists a neighborhood $\neighbor$ of $Z$ and a $C^2$ parametrization $\psi$ of $\xanifold\cap\neighbor$ with $\psi(0)=Z$ and $\tangent_\xanifold(Z)=\ima(\diff\psi(0))$. Write $X_i=\psi(s_i)$ for $i=1,2$. A Taylor expansion gives $X_1-X_2=\diff\psi(s_2)[s_1-s_2]+\complexity(\|s_1-s_2\|^2)$. Since $\diff\psi(s_2)[s_1-s_2]\in\tangent_\xanifold(X_2)$ and $\|s_1-s_2\|=\complexity(\|X_1-X_2\|)$, the estimate follows.
\end{proof}
The next lemma collects properties of the projection operator near a $C^2$ submanifold.

\begin{lemma}\label{lem:prox-regular}
Let $\xanifold$ be a $C^2$ embedded submanifold of $\eanifold$ and let $Z\in\xanifold$. Then there exist a neighborhood $\neighbor$ of $Z$, constants $\delta>0$ and $L>0$ such that the projection $\projection_\xanifold$ is single-valued and $L$-Lipschitz on $\{\tilde{X}\in\eanifold\mid\dist(\tilde{X},\xanifold\cap\neighbor)<\delta\}$. Moreover, for every $X\in\xanifold\cap\neighbor$, unit vector $v\in\normal_\xanifold(X)$, and $t\in(0,\delta)$, one has $\projection_\xanifold(X+tv)=X$.
\end{lemma}
\begin{proof}
This is a direct consequence of~\cite[Theorem~4.8]{federer1959curvature}.
\end{proof}
We then recall that a submanifold can be written as a graph over its tangent space.

\begin{lemma}[Local graph representation]\label{lem:local-graph}
Let $\xanifold$ be a $C^k$ ($k\ge 1$) embedded submanifold of $\eanifold$ and let $Z\in\xanifold$. Let $\tanL:=\tangent_\xanifold(Z)$. Then there exist a neighborhood $U_\tanL\subset\tanL$ of $0$ and a $C^k$ map $\mu:U_\tanL\to\tanL^\perp$ with $\mu(0)=0$ and $\diff\mu(0)=0$ such that $\xanifold\cap\neighbor = \{Z + u + \mu(u) \mid u\in U_\tanL\}$ for some neighborhood $\neighbor$ of $Z$.
\end{lemma}
\begin{proof}
Define $\pi:\eanifold\to\tanL$ by $\pi(X) := \projection_\tanL(X - Z)$. For all $v\in\tangent_\xanifold(Z) = \tanL$, we have $\diff(\pi|_\xanifold)(Z)[v] = \projection_\tanL(v) = v$, i.e., $\diff(\pi|_\xanifold)(Z) = \mathrm{id}_\tanL$. By the inverse function theorem, $\pi|_\xanifold$ is a local $C^k$ diffeomorphism near $Z$ with $\pi(Z)=0$. Let $\phi:U_\tanL\to\xanifold$ be its local inverse, such that $\pi(\phi(u))=u$ for all $u\in U_\tanL$. Define $\mu(u):=\phi(u)-Z-u$. Then $\mu(u)\in\tanL^\perp$ since $\projection_\tanL(\mu(u)) = \pi(\phi(u)) - u = 0$, and $\mu$ is $C^k$ with $\mu(0) = \phi(0)-Z = 0$. For $\diff\mu(0)$: since $\phi$ maps into $\xanifold$ and $\phi(0)=Z$, the image of $\diff\phi(0)$ lies in $\tangent_\xanifold(Z)=\tanL$; differentiating $\pi\circ\phi=\mathrm{id}$ gives $\projection_\tanL\circ\diff\phi(0)=\mathrm{id}_\tanL$, and thus $\diff\phi(0)=\mathrm{id}_\tanL$. It follows that $\diff\mu(0)[v]=\diff\phi(0)[v]-v=0$ for all $v\in\tanL$.
\end{proof}

\subsection{Regularities preserved under diffeomorphisms}\label{app:proof-clean}
We show in the following lemma that clean intersection is preserved under local diffeomorphisms.

\begin{lemma}\label{lem:clean_diffeo}
Let $\xanifold,\yanifold\subseteq\eanifold$ be $C^1$ embedded submanifolds that intersect cleanly at $Z\in\xanifold\cap\yanifold$, and let $\Psi:\eanifold\to\eanifold'$ be a $C^1$ diffeomorphism defined near $Z$. Then $\Psi(\xanifold)$ and $\Psi(\yanifold)$ intersect cleanly at $\Psi(Z)$ in $\eanifold'$.
\end{lemma}
\begin{proof}
Since $\Psi$ is a diffeomorphism, $\Psi(\xanifold\cap\yanifold) = \Psi(\xanifold)\cap\Psi(\yanifold)$ locally, which is a $C^1$ submanifold. Moreover, the chain rule gives $\tangent_{\Psi(\xanifold)}(\Psi(Z)) = \diff\Psi(Z)[\tangent_\xanifold(Z)]$ and similarly for $\yanifold$ and $\xanifold\cap\yanifold$. Since $\diff\Psi(Z)$ is an isomorphism, we have
\begin{align*}
\diff\Psi(Z)[\tangent_\xanifold(Z)] \cap \diff\Psi(Z)[\tangent_\yanifold(Z)] = \diff\Psi(Z)[\tangent_\xanifold(Z)\cap\tangent_\yanifold(Z)] = \diff\Psi(Z)[\tangent_{\xanifold\cap\yanifold}(Z)].
\end{align*}
The same argument holds for points around $Z$.
\end{proof}

The next lemma, noted as an exercise in~\cite[\S3]{drusvyatskiy2015intrinsictransversality}, reveals that intrinsic transversality is preserved under local diffeomorphisms; we restate it below for completeness.

\begin{lemma}\label{lem:IT_diffeo}
Let $\xanifold,\yanifold\subseteq\eanifold$ be $C^1$ embedded submanifolds that are intrinsically transversal at $Z\in\xanifold\cap\yanifold$, and let $\Psi:\eanifold\to\eanifold'$ be a $C^1$ local diffeomorphism defined near $Z$. Then $\Psi(\xanifold)$ and $\Psi(\yanifold)$ are intrinsically transversal at $\Psi(Z)$ in $\eanifold'$.
\end{lemma}

\section{Projection-based transversality and feasibility direction}\label{subsec:feasibility}
By~\eqref{eq:constructionofGnGt}, designing the feasibility direction $\Gh$ reduces to specifying $d(X)$. Since the goal of $\Gh$ is to drive the iterate $X\in\manifold$ towards $\hanifold$, a reasonable candidate is along the projection residual $r(X):=\projection_{\hanifold}(X)-X$, which points from $X$ towards $\hanifold$. However, recalling that $d(X)$ is required to lie in $\normal_{\hanifoldX}(X)$ to guarantee the orthogonality of $\Gh$ and $\Gf$, we instead approximate $r(X)$ by the \emph{Gauss--Newton direction}, which solves the linearized least-squares problem $\min_{d\in\normal_{\hanifoldX}(X)}\|h(X)+\diff h_X(d)\|^2$:
\begin{equation}\label{eq:GN_direction}
    d(X) = -\diff h_X^*\big(\diff h_X \diff h_X^*\big)^{-1} h(X).
\end{equation}
Note that the direction $d(X)$ lies in $\normal_{\hanifoldX}(X)=\ima(\diff h_X^*)$, and it approximates the projection residual $r(X)$ to first order.
\begin{lemma}\label{lem:GN_approx}
Under Assumption~\ref{assu:h}, let $Z\in\hanifold$. There exist a neighborhood $\neighbor$ of $Z$ and $C_d>0$ such that $\|d(X)-r(X)\| \le C_d\|r(X)\|^2$ for all $X\in\neighbor$.
\end{lemma}
\begin{proof}
Let $X_p = \projection_{\hanifold}(X)$ such that $r(X) = X_p - X$ and $h(X_p)=0$. The projection gives $X - X_p \in \normal_{\hanifold}(X_p) = \ima(\diff h_{X_p}^*)$, and thus $-r(X) = \diff h_{X_p}^*(v)$ for some $v\in\mbR^q$. A Taylor expansion yields $h(X) = h(X_p) + \diff h_{X_p}(X - X_p) + \complexity(\|r(X)\|^2) = \diff h_{X_p}\diff h_{X_p}^*(v) + \complexity(\|r(X)\|^2)$. Since $\diff h_X = \diff h_{X_p} + \complexity(\|r(X)\|)$, we obtain $d(X) = -\diff h_X^*(\diff h_X\diff h_X^*)^{-1}h(X) = -\diff h_{X_p}^*(v) + \complexity(\|r(X)\|^2)  = r(X) + \complexity(\|r(X)\|^2)$.
\end{proof}

Following~\eqref{eq:constructionofGnGt}, the feasibility direction $\Gh(X)$ is given as follows,
\begin{equation}\label{eq:Gn}
    \Gh(X) = \projection_{\tangent_{\manifold}(X)} (d(X)),
\end{equation}
the projected Gauss--Newton direction. An ensuing question is whether the projected component of $d(X)$ still pulls the iterates towards $\hanifold$. To answer this, we investigate the intersection geometry of $\hanifold$ and $\manifold$, taking into account the intrinsic transversality.

Intuitively, to guarantee that the projected $\Gh(X)$ inherits the ability of $d(X)$ to decrease $\dist(\cdot,\hanifold)$, the original $d(X)$ should retain sufficient norm along the tangent space of $\manifold$, motivating us to propose the following condition.

\begin{definition}[Projection-based transversality]
\label{def:PS-uniform}
Two manifolds $\xanifold,\yanifold\subseteq\eanifold$ satisfy \emph{projection-based transversality} at $Z\in\xanifold\cap\yanifold$ if
there exist a constant $\kappa'\in(0,1]$ and a neighborhood $\neighbor$ of $Z$ such that, denoting $u=(X-Y)/\|X-Y\|$,
\begin{enumerate}[leftmargin=2em]
\item[\emph{(i)}] for all $Y\in (\yanifold\cap \neighbor)\setminus \xanifold$ with $X=\projection_{\xanifold}(Y)$, one has $\dist\big(u, \normal_{\yanifold}(Y)\big) \ge \kappa'$;
\item[\emph{(ii)}] for all $X\in (\xanifold\cap \neighbor)\setminus \yanifold$ with $Y=\projection_{\yanifold}(X)$, one has $\dist\big(u,-\normal_{\xanifold}(X)\big) \ge \kappa'$.
\end{enumerate}
\end{definition}

Since the above condition constrains only the \emph{projection pairs} $(X,\projection_{\yanifold}(X))$ and $(\projection_{\xanifold}(Y),Y)$, it is a priori weaker than intrinsic transversality, which constrains \emph{all pairs} around the common point (see Definition~\ref{def:intrinsic_transversality}). Nevertheless, we then show in Theorem~\ref{thm:PB-implies-IT} that, for $C^2$ submanifolds, the two conditions are equivalent, before which the following lemma is presented as a preliminary.

\begin{lemma}\label{lem:normalization}
For all nonzero $a,b\in\eanifold$, $\big\|a/\|a\|-b/\|b\|\big\|\le 2\|a-b\|/\|a\|$.
\end{lemma}
\begin{proof}
It is seen from the equality $a/\|a\|-b/\|b\|=(a-b)/\|a\|+b(1/\|a\|-1/\|b\|)$. 
\end{proof}

\begin{theorem}\label{thm:PB-implies-IT}
Let $\xanifold$ and $\yanifold$ be $C^2$ embedded submanifolds of $\eanifold$, and let $Z\in\xanifold\cap\yanifold$.
Then $\xanifold$ and $\yanifold$ are intrinsically transversal at $Z$ if and only if they satisfy the projection-based transversality at $Z$.
\end{theorem}
\begin{proof}
For the ``only if'' part, assume that the intrinsic transversality holds with constant $\kappa$. Let $Y\in(\yanifold\cap\neighbor)\setminus\xanifold$ and $X=\projection_\xanifold(Y)$, $u=(X-Y)/\|X-Y\|$. If $X\notin\yanifold$, the optimality of the projection gives $Y-X\in\normal_\xanifold(X)$, and thus $u\in-\normal_\xanifold(X)$, which together with~\eqref{eq:IT_kappa} directly gives $\dist(u,\normal_\yanifold(Y))\ge\kappa$. If $X\in\xanifold\cap\yanifold$, then $X,Y\in\yanifold$ and Lemma~\ref{lem:chord-tangent} applied with base point $Y$ gives $\dist(X-Y,\tangent_\yanifold(Y))\le C_\yanifold\|X-Y\|^2$ for some $C_{\yanifold}>0$. Hence $\dist(u,\tangent_\yanifold(Y))\le C_\yanifold\|X-Y\|$ and $\dist(u,\normal_\yanifold(Y))\ge 1/2$ for $\|X-Y\|$ small enough. Consequently, condition~(i) holds with $\kappa'=\min\{\kappa,1/2\}$. The symmetric argument yields~(ii).

For the ``if'' part, assume that projection-based transversality holds with constant $\kappa'>0$. By Lemma~\ref{lem:prox-regular}, for $\sanifold\in\{\xanifold,\yanifold\}$, there exist a neighborhood $\neighbor$ of $Z$, $\delta>0$, and $L>0$ such that $\projection_\sanifold$ is single-valued and $L$-Lipschitz on $\{\tilde{X}\in\eanifold\mid\dist(\tilde{X},\sanifold\cap \neighbor)<\delta\}$, and $\projection_\sanifold(X+tv)=X$ for $X\in \sanifold\cap \neighbor$, unit vector $v\in\normal_\sanifold(X)$, and $t\in(0,\delta)$. Suppose, for contradiction, that intrinsic transversality fails at $Z$. Then for every $\kappa\in (0,1)$, there exist $X\in (\xanifold\cap \neighbor)\setminus\yanifold$ and $Y\in(\yanifold\cap \neighbor)\setminus\xanifold$ with
\begin{equation}\label{eq:IT_fail}
\dist(u,\normal_\yanifold(Y))<\kappa
\quad\text{and}\quad
\dist(u,-\normal_\xanifold(X))<\kappa,
\end{equation}
where $u=(X-Y)/\|X-Y\|$. We show that~\eqref{eq:IT_fail} indeed implies a violation of projection-based transversality for the projection pair $(X, Y^\star)$ with $Y^\star:=\projection_\yanifold(X)$.

Set $\varepsilon:=\dist(u,\normal_\yanifold(Y))$, $\tau:=\|X-Y\|$, and $w:=\projection_{\normal_\yanifold(Y)}(u)$. Since $\varepsilon < 1$, the vector $w$ is nonzero, and thus $v:=w/\|w\|$ is a well-defined unit vector in $\normal_\yanifold(Y)$. The orthogonal decomposition $u = w + (u - w)$ yields $\|w\| = \sqrt{1-\varepsilon^2}$ and $\|u - v\|^2 =  2 - 2\sqrt{1-\varepsilon^2} \le 2\varepsilon^2$.
Shrinking $\neighbor$ if necessary such that $\tau < \delta$, Lemma~\ref{lem:prox-regular} ensures $\projection_\yanifold(Y+\tau v)=Y$. Since $X = Y + \tau u$, the Lipschitz continuity of $\projection_\yanifold$ yields
$\|Y^\star-Y\|
=\|\projection_\yanifold(X)-\projection_\yanifold(Y+\tau v)\|
\le L\tau\|u-v\| \le \sqrt{2}\,L\tau\varepsilon$,
and consequently, $\|X-Y^\star\| \ge \tau - \|Y-Y^\star\| \ge \tau(1-\sqrt{2}L\varepsilon) \ge \tau/2$ once $\varepsilon$ is small enough. Let $u^\star:=(X-Y^\star)/\|X-Y^\star\|$. Lemma~\ref{lem:normalization} with $a=X-Y^\star$ and $b=X-Y$ gives $\|u^\star-u\|\le 4\|Y-Y^\star\|/\tau\le 4\sqrt{2}L\varepsilon$. The $1$-Lipschitz property of the distance to a subspace then gives $\dist(u^\star,-\normal_\xanifold(X))
\le \dist(u,-\normal_\xanifold(X))+\|u^\star-u\|
< \kappa + 4\sqrt{2}L\varepsilon
< (1+4\sqrt{2}L)\kappa$, where the inequalities follow from~\eqref{eq:IT_fail}. When $\kappa$ is small such that $(1+4\sqrt{2}L)\kappa<\kappa'$, it contradicts projection-based transversality for $(X,Y^\star)$.
\end{proof}

Applying Theorem~\ref{thm:PB-implies-IT} to $(\xanifold,\yanifold)=(\manifold,\hanifold)$, intrinsic transversality ensures that the projection direction from $\manifold$ onto $\hanifold$ is not nearly normal to $\manifold$, that is,
\[
\|\projection_{\tangent_{\manifold}(X)}\kh{r(X)/\|r(X)\|}\|  = \dist\kh{r(X)/\|r(X)\|,\normal_{\manifold}(X)} \ge \kappa^\prime.
\]
Recalling that the Gauss--Newton direction $d(X)$ approximates the projection direction $r(X)$, a similar property carries over to $\Gh(X) = \projection_{\tangent_\manifold(X)}(d(X))$.

\begin{proposition}\label{pro:angle}
Assume that $\hanifold$ and $\manifold$ are intrinsically transversal at $Z\in \hanifold\cap \manifold$ with constant $\kappa$. Then there exist a neighborhood $\neighbor$ of $Z$ and a constant $\kappa_0>0$ such that for all $X\in \manifold\cap \neighbor$, $\dist\kh{d(X),\,\normal_{\manifold}(X)} \ge \kappa_0\,\|d(X)\|$. Equivalently,
\begin{equation}\label{eq:gauss_dist_bound}
    \|\Gh(X)\| = \|\projection_{\tangent_\manifold(X)}(d(X))\|\ge \kappa_0\|d(X)\|.
\end{equation}
\end{proposition}
\begin{proof}
By Lemma~\ref{lem:prox-regular}, $\projection_{\hanifold}$ is single-valued and Lipschitz in a neighborhood $\neighbor$ of $Z$. Take $X\in\manifold\cap\neighbor\setminus\hanifold$ and set $X_p:=\projection_{\hanifold}(X)$, $r(X):=X_p - X$, $\hat{r}:=r(X)/\|r(X)\|$. The projection gives $X - X_p \in \normal_{\hanifold}(X_p)$, which implies that $-\hat{r}\in\normal_{\hanifold}(X_p)$.

We first show that $\hat{r}$ is bounded away from $\normal_{\manifold}(X)$. If $X_p\notin\manifold$, applying intrinsic transversality~\eqref{eq:IT_kappa} to the pair $(X, X_p)$ with $u = -\hat{r}$ yields $\dist(-\hat{r},-\normal_{\manifold}(X))\ge\kappa$, i.e., $\dist(\hat{r},\normal_{\manifold}(X))\ge\kappa$. If $X_p\in\manifold$, then both $X$ and $X_p$ lie on $\manifold$, and Lemma~\ref{lem:chord-tangent}, after shrinking $\neighbor$ if necessary, gives $\dist(\hat{r},\normal_{\manifold}(X))\ge 1/2$. In either case, we have
\begin{equation}\label{eq:rhat_angle}
\dist(\hat{r},\normal_{\manifold}(X))\ge\min\{\kappa,1/2\}=:\kappa_1.
\end{equation}

It remains to transfer this bound from the projection residual $r(X)$ to the Gauss--Newton direction $d(X)$. By Lemma~\ref{lem:GN_approx}, $\|d(X)-r(X)\|\le C\|r(X)\|^2$ for a constant $C>0$. Shrinking $\neighbor$ such that $\|d(X)-r(X)\|\le\frac{1}{2}\|r(X)\|$ ensures that $\|d(X)\|\ge\frac{1}{2}\|r(X)\|>0$. Writing $\hat{d}:=d(X)/\|d(X)\|$, Lemma~\ref{lem:normalization} gives $\|\hat{d}-\hat{r}\|\le 2\|d(X)-r(X)\|/\|d(X)\|\le 4C\|r(X)\|$. This, together with~\eqref{eq:rhat_angle} and the $1$-Lipschitz property of the distance to a subspace, yields $\dist(\hat{d},\normal_{\manifold}(X))\ge\kappa_1-4C\|r(X)\|$. Shrinking $\neighbor$ further such that $4C\|r(X)\|<\kappa_1/2$, we obtain $\dist(\hat{d},\normal_{\manifold}(X))\ge\kappa_1/2$. The conclusion follows by taking $\kappa_0:=\kappa_1/2$ and noting $\|\projection_{\tangent_{\manifold}(X)}({d})\| = \dist({d},\normal_{\manifold}(X))$.
\end{proof}

Proposition~\ref{pro:angle} shows that the projected Gauss--Newton direction preserves a certain component of the original $d(X)$, and thus $\Gh(X)$ serves as a descent direction for the feasibility measure $\dist(\cdot,\hanifold)$; the formal statement is deferred to Proposition~\ref{prop:delta_descent}.

\section{Clean intersection and optimality direction}\label{subsec:clean}
We now turn to the optimality direction. Recall from~\eqref{eq:constructionofGnGt} that constructing $\Gf$ reduces to choosing $\eta(X)\in\eanifold$ and projecting it onto $S(X)=\tangent_{\hanifoldX\cap\manifold}(X)$. A natural choice for decreasing the objective is the negative gradient $\eta(X) = -\nabla f(X)$, yielding
\begin{equation}\label{eq:Gt}
    \Gf(X) = \projection_{S(X)}(-\nabla f(X)).
\end{equation}

\subsection{Equivalence between intrinsic transversality and clean intersection}
The immediate obstacle is that the structure of the tangent cone $\tangent_{\hanifoldX\cap\manifold}(X)$ remains unclear, which impedes the computation of~\eqref{eq:Gt}. To circumvent it, we then investigate the relationship between $\tangent_{\hanifoldX\cap\manifold}(X)$ and the intersection $\tangent_{\hanifoldX}(X)\cap\tangent_{\manifold}(X)$, enlightened by the general inclusion~\eqref{eq:cone_oneside}. As a result, it is concluded in the following theorem that intrinsic transversality indeed implies clean intersection (see Definition~\ref{def:clean_intersection}) for smooth manifolds, answering the question posed in~\cite[\S8]{drusvyatskiy2015intrinsictransversality}.

\begin{theorem}\label{thm:IT_implies_clean}
Let $\xanifold,\yanifold\subseteq\eanifold$ be $C^1$ embedded submanifolds and let $Z\in\xanifold\cap\yanifold$. Then $\xanifold$ and $\yanifold$ are intrinsically transversal at $Z$ if and only if they intersect cleanly at $Z$, i.e., there exists a neighborhood $\neighbor$ of $Z$ such that
$\xanifold\cap \yanifold \cap \neighbor$ is a $C^1$ embedded submanifold and
\begin{equation}\label{eq:clean_tangent}
\tangent_{\xanifold\cap\yanifold}(X)=\tangent_{\xanifold}(X)\cap \tangent_{\yanifold}(X) \quad \text{for all } X\in\xanifold\cap\yanifold\cap\neighbor.
\end{equation}
\end{theorem}
\begin{proof}
The ``if'' part is shown in~\cite[\S3]{drusvyatskiy2015intrinsictransversality}, and thus it suffices to prove the converse. By an appropriate translation, we assume $Z=0$ and write $\tanL:=\tangent_{\xanifold}(0)$.
According to Lemma~\ref{lem:local-graph}, there exist a neighborhood $U_{\tanL}\subseteq\tanL$ of $0$ and a $C^1$ map $\mu:U_{\tanL}\to \tanL^\perp$ with $\xanifold\cap \neighbor_1=\hkh{\, u+\mu(u) \mid u\in U_{\tanL}\,}$, $\mu(0)=0$, and $\diff \mu(0)=0$ for a neighborhood $\neighbor_1$ of $Z$.
Define the $C^1$ map $\Psi:\eanifold\to\eanifold$ by $\Psi(u+n):=u+(n-\mu(u))$ for $u\in\tanL$, $n\in\tanL^\perp$. One readily verifies that $\Psi(0)=0$ and $\diff\Psi(0)=I$, and thus $\Psi$ is a local $C^1$ diffeomorphism near $0$. Moreover, $\Psi(u+\mu(u)) = u$ for all $u\in U_\tanL$, and thus $\Psi$ maps $\xanifold$ locally onto $\tanL$.
By Lemma~\ref{lem:IT_diffeo}, $\tanL$ and $\Psi(\yanifold)$ are intrinsically transversal at~$0$. It therefore suffices to prove clean intersection for $\tanL\cap\Psi(\yanifold)$. For brevity, we slightly abuse the notation by renaming $\Psi(\yanifold)$ as $\yanifold$ in the following analysis.

Set $\tanQ:=\tangent_{\yanifold}(0)$ and decompose $\eanifold$ into four orthogonal subspaces by $\eanifold=\kanifold\oplus\uanifold\oplus\vanifold\oplus\wanifold$, 
where $\kanifold:=\tanL\cap \tanQ$, $\tanL=\kanifold\oplus \uanifold$, $\tanQ=\kanifold\oplus \vanifold$, and $\wanifold:=(\tanL+\tanQ)^\perp$; we remark that the orthogonality between $\uanifold$ and $\vanifold$ can be guaranteed after applying an invertible linear transformation to $(\tanL,\yanifold)$ if needed. Points are written as $p=(k,u,v,w)$ accordingly. By Lemma~\ref{lem:local-graph} applied to $\yanifold$ with tangent space $\tanQ$, there exists a $C^1$ map $g:U_{\tanQ}\to \tanQ^\perp=\uanifold\oplus\wanifold$ with $g(0)=0$ and $\diff g(0)=0$ such that, writing $g(k,v)=(a(k,v),\,b(k,v))$ with $a:\kanifold\oplus\vanifold\to \uanifold$ and $b:\kanifold\oplus\vanifold\to\wanifold$,
\begin{equation}\label{eq:H_param}
\yanifold\cap \neighbor_3
=
\hkh{\, (k,\ a(k,v),\ v,\ b(k,v)) \mid (k,v)\in U_{\tanQ}\,}
\end{equation}
for a neighborhood $\neighbor_3$ of $0$. Throughout the proof, we identify the tuple $(k,u,v,w)\in\kanifold\oplus\uanifold\oplus\vanifold\oplus\wanifold$ with the element $k+u+v+w\in\eanifold$, according to the decomposition.

We claim that there exists $\varepsilon>0$ such that
\begin{equation}\label{eq:b_vanish}
b(k,0)=0 \quad \text{for all } k\in\kanifold \text{ with } \norm{k}<\varepsilon.
\end{equation}
Suppose otherwise: there exists a sequence $k_j\to 0$ with $b(k_j,0)\neq 0$ for all $j$. Setting $v=0$ in~\eqref{eq:H_param} produces the pairs $x_j:=(k_j,\ a(k_j,0),\ 0,\ 0)\in \tanL$, $y_j:=(k_j,\ a(k_j,0),\ 0,\ b(k_j,0))\in \yanifold$.
Since $b(k_j,0)\neq 0$, the point $y_j$ has a nonzero $\wanifold$-component, and thus $y_j\notin\tanL$. Moreover, $\projection_\tanQ(x_j) = (k_j,0) = \projection_\tanQ(y_j)$, and since $\projection_\tanQ|_\yanifold$ is a local diffeomorphism, $y_j$ is the unique point of $\yanifold$ near $0$ projecting to $(k_j,0)$; as $x_j\neq y_j$, we have $x_j\notin\yanifold$. The difference $x_j-y_j=(0,0,0,-b(k_j,0))$ lies in $\wanifold\subseteq\normal_\tanL(x_j)=\vanifold\oplus\wanifold$, and thus the unit vector $u_j:=(x_j-y_j)/\norm{x_j-y_j}$ satisfies $\dist(u_j,-\normal_{\tanL}(x_j))=0$. On the other hand, since $\yanifold$ is locally the graph $t\mapsto t+g(t)$ over $\tanQ$, the tangent space at $y_j$ is $\tangent_{\yanifold}(y_j) = \{\dot{t}+\diff g(k_j,0)[\dot{t}] \mid \dot{t}\in\tanQ\}$. Consider the vector $n_j:=-\diff g(k_j,0)^*u_j+u_j$, and we can verify that $\innerp{n_j,\,\dot t+\diff g(k_j,0)[\dot t]}=0$ for all $\dot{t}\in\tanQ$, thereby $n_j\in\normal_{\yanifold}(y_j)$. Noticing that $\diff g(k_j,0)\to 0$, we have $\norm{n_j-u_j}\to 0$, and thus $\lim_{j\to\infty}\max\hkh{\dist(u_j,\normal_{\yanifold}(y_j)),\,\dist(u_j,-\normal_{\tanL}(x_j))} = 0$, which contradicts the intrinsic transversality of $\tanL\cap\yanifold$ at $0$.

According to~\eqref{eq:H_param} and~\eqref{eq:b_vanish}, the intersection $\tanL\cap\yanifold$ is parametrized by
$F:\kanifold\to\eanifold:\,\xi\mapsto (\xi,\ a(\xi,0),\ 0,\ 0)$ in the sense that $\tanL\cap\yanifold\cap\neighbor_4=\{F(\xi)\mid \xi\in\kanifold,\,\|\xi\|<\varepsilon\}$ for a neighborhood $\neighbor_4$ of $0$. Note that $F$ is $C^1$ with the injective differential $\diff F(0):\xi\mapsto(\xi,0,0,0)$ (since $\diff a(0,0)=0$). By~\cite[Theorem~4.25]{lee2012manifolds}, $F$ is a local embedding, and thus $\tanL\cap\yanifold$ is a $C^1$ submanifold of dimension $\dime(\kanifold)$ near the origin.

It remains to verify $\tangent_{\tanL\cap\yanifold}(p) = \tangent_\tanL(p)\cap\tangent_\yanifold(p)$ at every nearby intersection point $p=F(\xi)$. The inclusion ``$\subseteq$'' holds in general by~\eqref{eq:cone_oneside}. For ``$\supseteq$'', take any $(\dot k,\dot u,\dot v,\dot w)\in\tangent_\tanL(p)\cap\tangent_\yanifold(p)$. Membership in $\tangent_{\tanL}(p) = \kanifold\oplus\uanifold$ forces $\dot v=0$ and $\dot w=0$. Differentiating~\eqref{eq:H_param}, $(\dot k,\dot u,\dot v,\dot w)\in\tangent_\yanifold(p)$ requires $\dot u = \partial_k a(\xi,0)[\dot k]+\partial_v a(\xi,0)[\dot v]$ and $\dot w = \partial_k b(\xi,0)[\dot k]+\partial_v b(\xi,0)[\dot v]$. Substituting $\dot v=0$ gives $\dot u = \partial_k a(\xi,0)[\dot k]$ and $\dot w = \partial_k b(\xi,0)[\dot k]$. Since~\eqref{eq:b_vanish} implies $\partial_k b(\xi,0)=0$, the condition $\dot w = 0$ is satisfied, and the vector $(\dot k,\,\partial_k a(\xi,0)[\dot k],\,0,\,0) = \diff F(\xi)[\dot k]$ lies in $\ima(\diff F(\xi)) = \tangent_{\tanL\cap\yanifold}(p)$, establishing the ``$\supseteq$''. Applying Lemma~\ref{lem:clean_diffeo} through $\Psi^{-1}$ completes the proof.
\end{proof}

\subsection{Projection onto the tangent space} Recalling from Assumption~\ref{assu:IT} that in the neighborhood $\kanifold$, the manifolds $\hanifoldX$ and $\manifold$ are intrinsically transversal at the common point $X$. Therefore, applying Theorem~\ref{thm:IT_implies_clean} to the pair $(\hanifoldX,\manifold)$ gives the formula of the tangent cone~\eqref{eq:tangent_decom}.
This identification reveals that $S(X)=\tangent_{\hanifoldX\cap\manifold}(X)$ is a linear subspace, determined by $\tangent_\manifold(X)$ and $\tangent_{\hanifoldX}(X)$. Subsequently, the projection of any $\xi\in\eanifold$ onto $S(X)$ admits a closed-form expression; see the following proposition.

\begin{proposition}\label{pro:proj_SX}
Assume that $\hanifoldX$ and $\manifold$ are intrinsically transversal at
$X$ and let $\Phi_\manifold(X):=\projection_{\tangent_\manifold(X)}\circ\diff h_X^*$.
Then for all $\xi\in\eanifold$, it holds that
\begin{equation}\label{eq:proj_SX}
\projection_{S(X)}(\xi)
= \projection_{\tangent_\manifold(X)}(\xi)
  - \Phi_\manifold(X)
    \big(\diff h_X\circ\Phi_\manifold(X)\big)^{\dagger}
    \diff h_X\big(\projection_{\tangent_\manifold(X)}(\xi)\big).
\end{equation}
If additionally, $\hanifoldX$ and $\manifold$ are transversal at $X$,
then $\diff h_X\circ\Phi_\manifold(X)$ is invertible and the Moore--Penrose pseudoinverse reduces to the ordinary inverse.
\end{proposition}
\begin{proof}
We note that $\tangent_{\hanifoldX}(X) = \ker(\diff h_X)$ and $\normal_{\hanifoldX}(X) = \ima(\diff h_X^*)$. By Theorem~\ref{thm:IT_implies_clean}, the clean intersection~\eqref{eq:tangent_decom} holds. This implies that $S(X)$ is a linear subspace of $\tangent_\manifold(X)$, and thus $\projection_{S(X)}(\xi) = \projection_{S(X)}(\bar\xi)$ with $\bar\xi:=\projection_{\tangent_\manifold(X)}(\xi)$. Therefore, it suffices to project $\bar\xi\in\tangent_\manifold(X)$ onto $S(X)$.

We claim that $\tangent_\manifold(X) = S(X) \oplus \ima(\Phi_\manifold(X))$. To see this, denote the restricted linear map $L:=\diff h_X|_{\tangent_\manifold(X)}:\tangent_\manifold(X)\to\mbR^q$ with $S(X) = \ker(L)$. The adjoint $L^*:\mbR^q\to\tangent_\manifold(X)$ is characterized by $\innerp{L\eta, v} = \innerp{\eta, L^*v}$ for all $\eta\in\tangent_\manifold(X)$ and $v\in\mbR^q$. Since $\innerp{\diff h_X(\eta), v} = \innerp{\eta, \diff h_X^*(v)} = \innerp{\eta, \projection_{\tangent_\manifold(X)}(\diff h_X^*(v))}$, where the last equality uses $\eta\in\tangent_\manifold(X)$, we identify $L^*(v) = \Phi_\manifold(X)(v)$, and thus $\ima(\Phi_\manifold(X)) = \ima(L^*)$. The theory of linear algebra indicates that $\ker(L)$ and $\ima(L^*)$ are orthogonal complements within $\tangent_\manifold(X)$. Hence, the projection of $\bar\xi$ onto $S(X)$ is $\bar\xi - \projection_{\ima(\Phi_\manifold(X))}(\bar\xi)$.

To compute $\projection_{\ima(\Phi_\manifold(X))}(\bar\xi)$, write it as $\Phi_\manifold(X)\lambda$ for some $\lambda\in\mbR^q$. Since $\bar\xi - \Phi_\manifold(X)\lambda \in S(X) = \ker(\diff h_X|_{\tangent_\manifold(X)})$, we have $\diff h_X(\Phi_\manifold(X)\lambda) = \diff h_X(\bar\xi)$. The right-hand side satisfies $\diff h_X(\bar\xi) \in \ima(\diff h_X|_{\tangent_\manifold(X)}) = \ima(\diff h_X\circ\Phi_\manifold(X))$, and thus the system is consistent. A representative solution is $\lambda = (\diff h_X\circ\Phi_\manifold(X))^\dagger\diff h_X(\bar\xi)$, which is then substituted into the projection $\projection_{S(X)}(\bar\xi) = \bar\xi - \Phi_\manifold(X)\lambda$ to obtain~\eqref{eq:proj_SX}.
Under transversality, $\tangent_\manifold(X)+\ker(\diff h_X)=\eanifold$, equivalent to $\normal_\manifold(X)\cap\ima(\diff h_X^*)=\{0\}$. This means $\Phi_\manifold(X) = \projection_{\tangent_\manifold(X)}\circ\diff h_X^*$ is injective, and thus $\diff h_X\circ\Phi_\manifold(X)$ is invertible.
\end{proof}

Proposition~\ref{pro:proj_SX} provides a recipe for computing the projection onto the tangent space of the manifold intersection under intrinsic transversality, and in particular yields the optimality direction $\Gf(X)$~\eqref{eq:Gt}. We then apply it to the case when $\hanifold$ and $\manifold$ are the hyperboloid and the low-rank manifold, respectively; see Appendix~\ref{app:proj-hyperboloid}.

\section{Geometric method and convergence analysis}\label{sec:convergence}
Combining the feasibility direction~\eqref{eq:GN_direction} and the optimality direction~\eqref{eq:Gt}, we now present the geometric method via orthogonal tangent directions in Algorithm~\ref{alg:GOTD}.

\begin{algorithm}[htbp]
\renewcommand{\thealgorithm}{1}
\caption{Geometric method via Orthogonal Tangent Directions (GOTD)}\label{alg:GOTD}
\renewcommand{\algorithmicrequire}{\textbf{Input:}}
\renewcommand{\algorithmicensure}{\textbf{Output:}}
\begin{algorithmic}[1]
\REQUIRE Initial point $X_0\in\manifold$,  iteration number $K$, step sizes $\{\alpha_k,\beta_k\}_{k=0}^{K-1}$
\FOR{$k=0,1,\ldots,K-1$}
    \STATE Gauss--Newton direction $d_k = -\diff h_{X_k}^*\big(\diff^{} h_{X_k} \diff h_{X_k}^*\big)^{-1}h(X_k)$
    \STATE Feasibility direction $\Gh(X_k) = \projection_{\tangent_{\manifold}(X_k)}(d_k)$
    \STATE Optimality direction $\Gf(X_k)  = \projection_{S_k}(-\nabla f(X_k))$, $S_k :=\tangent_{\hanifold_{X_k}}(X_k)\cap\tangent_{\manifold}(X_k)$\hspace{-5mm}
    \STATE Update $X_{k+1} = \retrac_{X_k}^{\manifold}\big(\alpha_k \Gh(X_k) + \beta_k \Gf(X_k)\big)$
\ENDFOR
\ENSURE $X_K$
\end{algorithmic}
\end{algorithm}

Noting that the iterates generated by Algorithm~\ref{alg:GOTD} are not necessarily on the set $\hanifold$, we carry out the convergence analysis in a neighborhood of $\hanifold$ and show that the iterates approach $\hanifold$ asymptotically. In light of the tubular neighborhood theorem~\cite[Theorem~6.24]{lee2012manifolds}, we make the following assumption.
\begin{assumption}[Tubular neighborhood]\label{assu:tubular_new}
There exists $\Delta>0$ such that $\projection_\hanifold$ is single-valued and Lipschitz on the tube $\kanifold:=\{X\in\eanifold \mid \delta(X)<\Delta\}$, where $\delta(X):=\dist(X,\hanifold)$. In addition, $f_{\inf}:=\inf\{f(X) \mid X\in\kanifold\cap\manifold\}>-\infty$.
\end{assumption}

We also require some regularity conditions on $h$, commonly adopted in equality-constrained optimization~\cite{nocedal2006numerical,schechtman2023ODCGM,ablin2024infeasible}.

\begin{assumption}\label{assu:bound_nablah}
    There exist constants $\sigma,\mu,L_h>0$ such that $\sigma^2 I_q \preceq \diff h_X \diff h_X^* \preceq \mu^2 I_q$ and $\|\diff h_X-\diff h_Y\|\le L_h\|X-Y\|$ for all $X,Y\in \kanifold$.
\end{assumption}

\begin{assumption}\label{assu:tubular_proj}
    The projection $\projection_{\hanifold}$ is $C^1$ on $\kanifold$ with Lipschitz derivative: there exists $L_p>0$ such that $\|\diff(\projection_\hanifold)_X-\diff(\projection_\hanifold)_Y\|\le L_p\|X-Y\|$ for all $X,Y\in\kanifold$.
\end{assumption}

Turning to $\manifold$, the next assumption takes into account the approximation property of the retraction and the Riemannian smoothness of $f$, both standard in Riemannian optimization~\cite{boumal2019global}.

\begin{assumption}\label{assu:manifold_reg}
    There exist $L_R,L_f,M_f,\rho_\manifold>0$ such that for all $X\in \kanifold\cap\manifold$ and $\eta\in\tangent_{\manifold}(X)$ with $\|\eta\|\le\rho_\manifold$, $\|X+\eta - \retrac^{\manifold}_X(\eta)\|\le {L_R}\|\eta\|^2$, $\|\grad_{\manifold}f(X)\|\le M_f$, and
    \begin{equation}\label{eq:retrac_descent}
    f\big(\retrac_X^{\manifold}(\eta)\big)
    \le
    f(X)+\langle \grad_{\manifold}f(X),\,\eta\rangle
    +\frac{L_f}{2}\|\eta\|^2.
    \end{equation}
\end{assumption}

Assumptions~\ref{assu:bound_nablah}--\ref{assu:tubular_proj} and Assumption~\ref{assu:manifold_reg} consider the regularities of $\hanifold$ and $\manifold$, respectively. We then turn to the intersection geometry. Specifically, recall that Assumption~\ref{assu:IT} concerns the intrinsic transversality of $\hanifoldX\cap\manifold$ in $\kanifold$; this, in views of Proposition~\ref{pro:angle}, justifies the following assumption with a uniform constant $\kappa_0>0$.

\begin{assumption}\label{assu:angle}
There exists $\kappa_0>0$ such that for all $X\in\kanifold\cap\manifold$,
\begin{equation}\label{eq:angle_bound}
    \|\projection_{\tangent_\manifold(X)}(d(X))\|\ge \kappa_0\,\|d(X)\|.
\end{equation}
\end{assumption}

\subsection{Auxiliary lemmas}\label{subsec:comparison}
This section collects properties regarding the level sets defined by $h$ in $\kanifold$, and thus all the lemmas below are based on Assumptions~\ref{assu:h}, \ref{assu:tubular_new}, \ref{assu:bound_nablah}, and~\ref{assu:tubular_proj}. We write $r(X)=\projection_{\hanifold}(X)-X$ as the projection residual. The following lemma presents the error bound condition of $h$ and shows that the Gauss--Newton direction $d(X)$~\eqref{eq:GN_direction} approximates $r(X)$ to second order.

\begin{lemma}\label{lem:delta_h_d}
There exist constants $c_h,C_h,C_d>0$ such that $c_h\,\delta(X)\le \|h(X)\|\le C_h\,\delta(X)$ and $\|d(X)-r(X)\|\le C_d\,\delta(X)^2$ for all $X\in\kanifold$.
\end{lemma}
\begin{proof}
Let $Y:=\projection_{\hanifold}(X)$, such that $h(Y)=0$. Since $h$ is $C^2$, we have $h(X) = h(Y) + \diff h_Y(X-Y) + R$ where $\|R\| \le \frac{L_h}{2}\|X-Y\|^2 = \frac{L_h}{2}\delta(X)^2$. Noting that $X - Y = -r(X)$, we have $h(X) = -\diff h_Y[r(X)] + R$. Since the projection gives $r(X)\in\normal_\hanifold(Y) = \ima(\diff h_Y^*)$, Assumption~\ref{assu:bound_nablah} applied to $\diff h_Y|_{\ima(\diff h_Y^*)}$ yields $\sigma\delta(X) \le \|\diff h_Y[r(X)]\| \le \mu\delta(X)$, and thus $\|h(X)\| \le \mu\delta(X) + \frac{L_h}{2}\delta(X)^2$ and $\|h(X)\| \ge \sigma\delta(X) - \frac{L_h}{2}\delta(X)^2$.
Shrinking $\Delta$ such that $\frac{L_h}{2}\Delta \le \sigma/2$ yields the inequalities with $c_h = \sigma/2$ and $C_h = \mu + \frac{L_h}{2}\Delta$. The second estimate follows from the analysis of Lemma~\ref{lem:GN_approx}.
\end{proof}

The next lemma relates the magnitude of the feasibility direction $\Gh$ with the distance to the feasible region $\hanifold$.
\begin{lemma}\label{lem:Gn_delta}
There exist constants $c_n,C_n>0$ such that for all $X\in\kanifold$,
\begin{equation}\label{eq:Gn_delta}
    c_n\,\delta(X)\le \|\Gh(X)\|\le C_n\,\delta(X).
\end{equation}
\end{lemma}
\begin{proof}
Let $\Phi(X):=\diff h_X^*(\diff h_X\diff h_X^*)^{-1}$. By Assumption~\ref{assu:bound_nablah}, $\|y\|/\mu \le \|\Phi(X)y\| \le \|y\|/\sigma$ for all $y\in\mbR^q$. For the lower bound, Assumption~\ref{assu:angle} gives $\|\Gh(X)\| \ge \kappa_0\|d(X)\|$. Since $\|d(X)\| = \|\Phi(X)h(X)\| \ge \|h(X)\|/\mu \ge (c_h/\mu)\delta(X)$ by Lemma~\ref{lem:delta_h_d}, we obtain $c_n = \kappa_0 c_h/\mu$. For the upper bound, $\|\Gh(X)\| \le \|d(X)\| = \|\Phi(X)\|\|h(X)\| \le (C_h/\sigma)\delta(X)$, admitting $C_n = C_h/\sigma$.
\end{proof}

We then consider the derivative of the map $r$; the first observation is an identity that simplifies the inner product involving $\diff r$.
\begin{lemma}\label{lem:Dr_inner}
For all $X\in\kanifold$ and $\eta\in\eanifold$, one has $\innerp{r(X),\diff r_X[\eta]}=-\innerp{r(X),\eta}$.
\end{lemma}
\begin{proof}
Since $\diff r_X=\diff(\projection_{\hanifold})_X-I$ and $\projection_\hanifold$ maps into $\hanifold$, we have $\diff(\projection_{\hanifold})_X[\eta]\in\tangent_{\hanifold}(\projection_{\hanifold}(X))$. The identity then follows from $r(X)\in\normal_{\hanifold}(\projection_{\hanifold}(X))$.
\end{proof}

This section ends with a bound on the operator norm of $\diff r$, together with a sharper estimate along directions tangent to $\hanifold$.
\begin{lemma}\label{lem:Dr_tangent}
For all $X\in\kanifold$ with $Y:=\projection_\hanifold(X)$, $\|\diff r_X\|\le 1+L_p\,\Delta$ and $\|\diff r_X[v]\|\le L_p\,\delta(X)\,\|v\|$ for all $v\in\tangent_\hanifold(Y)$.
If additionally Assumption~\ref{assu:bound_nablah} holds, then there exists a constant $C_r>0$ such that $\|\diff r_X[\Gf]\|\le C_r\,\delta(X)\,\|\Gf\|$ for all $\Gf\in S(X)$.
\end{lemma}
\begin{proof}
Since $\projection_\hanifold|_\hanifold = \mathrm{Id}$, $\diff(\projection_\hanifold)_Y$ acts as the identity on $\tangent_\hanifold(Y)$ and vanishes on $\normal_\hanifold(Y)$, giving $\|\diff r_Y\|=1$. Assumption~\ref{assu:tubular_proj} then yields $\|\diff r_X\|\le 1+L_p\delta\le 1+L_p\Delta$. For $v\in\tangent_\hanifold(Y)$, $\diff r_Y[v]=0$, and thus $\|\diff r_X[v]\| = \|\diff r_X[v]-\diff r_Y[v]\| \le L_p\delta\|v\|$.

For the $C_r$ bound, decompose $\Gf = \Gf^0 + \Gf^\perp$ with $\Gf^0 := \projection_{\ker(\diff h_Y)}(\Gf)$ and $\Gf^\perp := \Gf - \Gf^0 \in \ker(\diff h_Y)^\perp$. The observation $\diff h_X[\Gf] = 0$ and the Lipschitz property of $\diff h$ indicate that $\|\diff h_Y[\Gf]\| \le L_h\delta\|\Gf\|$; as $\diff h_Y[\Gf^\perp] = \diff h_Y[\Gf]$, it follows that $\|\Gf^\perp\| \le (L_h/\sigma)\delta\|\Gf\|$. Since $\Gf^0\in\tangent_\hanifold(Y)$, the tangent bound gives $\|\diff r_X[\Gf^0]\| \le L_p\delta\|\Gf\|$. Moreover, note that $\Gf^\perp \in \ker(\diff h_Y)^\perp = \normal_\hanifold(Y)$. Hence, $\diff(\projection_\hanifold)_Y[\Gf^\perp] = 0$ and $\diff r_Y[\Gf^\perp] = -\Gf^\perp$. By Assumption~\ref{assu:tubular_proj}, $\|\diff r_X[\Gf^\perp] - \diff r_Y[\Gf^\perp]\| \le L_p\delta\|\Gf^\perp\|$, implying that $\|\diff r_X[\Gf^\perp]\| \le (1+L_p\Delta)\|\Gf^\perp\| \le (1+L_p\Delta)(L_h/\sigma)\delta\|\Gf\|$. The triangle inequality concludes that $C_r$ can be chosen as  $L_p + (1+L_p\Delta)L_h/\sigma$.
\end{proof}

\subsection{One-step estimates of feasibility and optimality}\label{subsec:descent}
We then derive the main ingredients of the analysis: one-step estimates for the feasibility measure $\delta$ and the objective value $f$ under the update rule~\eqref{eq:general_updateX}. We begin with the following proposition, which reveals that the feasibility direction $\Gh$ drives the decrease in $\delta$, while the optimality direction $\Gf$ introduces only higher-order perturbations.
\begin{proposition}\label{prop:delta_descent}
Under Assumptions~\ref{assu:h}--\ref{assu:angle}, there exist $\bar\alpha,\bar\beta,a_1,a_2,a_3,a_4>0$ such that for all $X\in\kanifold\cap\manifold$, $0<\alpha\le\bar\alpha$, and $0<\beta\le\bar\beta$, the iterate $X_+=\retrac_X^{\manifold}(\alpha\,\Gh(X)+\beta\,\Gf(X))$ satisfies
\begin{equation}\label{eq:delta_descent}
    \delta(X_{+})
    \le
    (1 - a_1\alpha + a_2\alpha^2)\,\delta(X)
    +a_3\,\beta\,\delta(X)\,\|\Gf(X)\|
    +a_4\,\beta^2\|\Gf(X)\|^2.
\end{equation}
\end{proposition}
\begin{proof}
Write $\eta:=\alpha\,\Gh+\beta\,\Gf$, $\delta:=\delta(X)$, $\delta_+:=\delta(X_+)$, $Y:=\projection_\hanifoldX$, and $\kappa_1:=(\kappa_0 c_h/(2\mu))^2$. We restrict $\Delta \le \min\{\kappa_0 c_h/(2\mu C_d),\, \kappa_1/(4C_d)\}$ and set $\bar\alpha = \min\{1/(2\kappa_1),\, \rho'/(2C_n\Delta)\}$, $\bar\beta = \rho'/(2M_f)$ with $\rho':=\min\{\rho_\manifold,\,1/L_R\}$, so that $\|\eta\| \le \alpha C_n\Delta + \beta M_f \le \rho'$ and Assumption~\ref{assu:manifold_reg} applies: $X_{+}=X+\eta+e_R$ with $\|e_R\|\le L_R\|\eta\|^2$. A Taylor expansion of $r$ at $X$ applied to the displacement $\eta+e_R$ gives $\delta_+ \le \|r(X)+\diff r_X[\eta+e_R]\| + \frac{L_p}{2}\|\eta+e_R\|^2$. By Lemma~\ref{lem:Dr_tangent}, $\|\diff r_X\|\le 1+L_p\Delta$. In addition, $\|e_R\|\le L_R\|\eta\|^2\le\|\eta\|$ (by $\|\eta\|\le 1/L_R$) implies $\|\eta+e_R\|\le 2\|\eta\|$, and thus
\begin{equation}\label{eq:delta_triangle}
\delta_+ \le \|r(X)+\diff r_X[\eta]\|+c_R\|\eta\|^2, \qquad\text{with}\ c_R := (1+L_p\Delta)L_R + 2L_p.
\end{equation}
Expanding $\|r+\diff r[\eta]\|^2$ according to Lemma~\ref{lem:Dr_inner},
\begin{equation}\label{eq:r_expand}
\|r+\diff r[\eta]\|^2 = \delta^2 - 2\alpha\innerp{r,\Gh} - 2\beta\innerp{r,\Gf} + \|\diff r[\eta]\|^2.
\end{equation}
We estimate the terms on the right. For the term involving $\Gh$, we note that $\Gh=\projection_{\tangent_{\manifold}(X)}(d)$ reveals that $\innerp{r,\Gh}=\innerp{\projection_{\tangent_\manifold(X)}(r),d}$, and writing $d = r + (d-r)$ with $\|d-r\|\le C_d\delta^2$ (Lemma~\ref{lem:delta_h_d}) yields
$\innerp{r,\Gh} = \|\projection_{\tangent_\manifold(X)}(r)\|^2 + \innerp{\projection_{\tangent_\manifold(X)}(r),d-r} \ge \|\projection_{\tangent_\manifold(X)}(r)\|^2 - C_d\delta^3$.
The triangle inequality and Assumption~\ref{assu:angle} give $\|\projection_{\tangent_\manifold(X)}(r)\| \ge \kappa_0\|d\| - C_d\delta^2 \ge (\kappa_0 c_h/\mu - C_d\delta)\delta$; restricting $\Delta \le \kappa_0 c_h/(2\mu C_d)$ ensures $\|\projection_{\tangent_\manifold(X)}(r)\| \ge \kappa_0 c_h\delta/(2\mu)$, and thus
$\innerp{r,\Gh} \ge \kappa_1\delta^2 - C_d\delta^3$.
For the term involving $\Gf$, the orthogonality $\innerp{d,\Gf}=0$ (Lemma~\ref{lem:orth_GtGn}) gives $|\innerp{r,\Gf}| = |\innerp{r-d,\Gf}| \le C_d\delta^2\|\Gf\|$. For the last term, Lemma~\ref{lem:Gn_delta} and Lemma~\ref{lem:Dr_tangent} yield $\|\diff r[\eta]\|^2\le 2(1+L_p\Delta)^2C_n^2\alpha^2\delta^2+2C_r^2\beta^2\delta^2\|\Gf\|^2$.
Substituting the estimates into~\eqref{eq:r_expand} shows $\|r+\diff r[\eta]\|^2 \le \delta^2(1 - t)$ with $t := 2\alpha(\kappa_1 - C_d\delta) - 2\beta C_d\|\Gf\| - 2(1+L_p\Delta)^2C_n^2\alpha^2 - 2C_r^2\beta^2\|\Gf\|^2$. The restriction on $\Delta$ ensures $\kappa_1 - C_d\delta \ge \kappa_1/2$, and $\bar\alpha \le 1/(2\kappa_1)$ gives $2\alpha\kappa_1 \le 1$, hence $t \le 1$. Applying $\sqrt{1-t}\le 1-t/2$ leads to $\|r+\diff r[\eta]\| \le \delta - \kappa_1\alpha\delta + C_d\alpha\delta^2 + C_d\beta\delta\|\Gf\| + (1+L_p\Delta)^2C_n^2\alpha^2\delta + C_r^2\beta^2\delta\|\Gf\|^2$.
Adding $c_R\|\eta\|^2 \le 2c_R\alpha^2C_n^2\delta^2 + 2c_R\beta^2\|\Gf\|^2$ from~\eqref{eq:delta_triangle} and using $\delta\le\Delta$ to replace each $\delta^2$ by $\Delta\delta$, we conclude with $a_1 = \tfrac{3\kappa_1}{4}$, $a_2 = (1+L_p\Delta)^2C_n^2 + 2c_RC_n^2\Delta$, $a_3 = C_d$, and $a_4 = C_r^2\Delta + 2c_R$.
\end{proof}

We then resort to the Riemannian smoothness assumption over $\manifold$, proving that the optimality direction $\Gf$ delivers the descent property of the objective $f$.

\begin{proposition}\label{prop:f_descent}
Under Assumptions~\ref{assu:h}--\ref{assu:angle}, let $X\in\kanifold\cap\manifold$, $0<\alpha\le\bar\alpha$, $0<\beta\le\bar\beta$ (with $\bar\alpha,\bar\beta$ from Proposition~\ref{prop:delta_descent}), and $X_+=\retrac_X^{\manifold}(\alpha\,\Gh(X)+\beta\,\Gf(X))$. Then
\begin{equation}\label{eq:f_descent}
    f(X_+)-f(X)
    \le
    -\beta\Big(1-\frac{L_f}{2}\beta\Big)\|\Gf(X)\|^2
    +\alpha\,M_f\,\|\Gh(X)\|
    +\frac{L_f}{2}\alpha^2\|\Gh(X)\|^2.
\end{equation}
\end{proposition}
\begin{proof}
Let $\eta:=\alpha\,\Gh+\beta\,\Gf$. The constructed $\bar\alpha$ and $\bar\beta$ ensure that $\|\eta\|\le\rho_\manifold$, and Assumption~\ref{assu:manifold_reg} produces $f(X_+) \le f(X) + \innerp{\grad_\manifold f(X),\eta} + \frac{L_f}{2}\|\eta\|^2$. For the inner product, $\Gf = \projection_{S(X)}(-\grad_\manifold f(X))\in\tangent_\manifold(X)$ yields $\innerp{\grad_\manifold f, \Gf} = -\|\Gf\|^2$, while $|\innerp{\grad_\manifold f, \Gh}| \le M_f\|\Gh\|$. The orthogonality $\innerp{\Gh,\Gf}=0$ (Lemma~\ref{lem:orth_GtGn}) gives $\|\eta\|^2 = \alpha^2\|\Gh\|^2+\beta^2\|\Gf\|^2$. Substitute the terms and rearrange them to yield~\eqref{eq:f_descent}.
\end{proof}

Drawing on Proposition~\ref{prop:delta_descent}, it is shown that the iterates are confined to the tube $\kanifold$ under appropriate step sizes.

\begin{lemma}\label{lem:invariant_tube}
Under Assumptions~\ref{assu:h}--\ref{assu:angle}, let $0<\Delta_0<\Delta$.
There exist constants $\bar\alpha'$, $\bar\beta'$, $\tau>0$ such that for all
$X\in\kanifold\cap\manifold$ with $\delta(X)\le\Delta_0$,
$0<\alpha\le\bar\alpha'$, $0<\beta\le\bar\beta'$, and $\beta\le\tau\alpha$,
the iterate $X_+=\retrac_X^\manifold(\alpha \Gh+\beta \Gf)$
satisfies $\delta(X_+)\le\Delta_0$.
\end{lemma}
\begin{proof}
Let $\bar\alpha,\bar\beta,a_1,\ldots,a_4$ be as in Proposition~\ref{prop:delta_descent}. Using $\|\Gf\|\le M_f$ and $\beta\le\tau\alpha$ in~\eqref{eq:delta_descent} gives $a_3\beta\delta\|\Gf\| \le a_3\tau\alpha\delta M_f$ and $a_4\beta^2\|\Gf\|^2 \le a_4\tau^2\alpha^2 M_f^2$. Set $\tau = a_1/(4a_3 M_f)$ so that $a_3\tau M_f = a_1/4$, and define
\[
\bar\alpha' = \min\big\{\bar\alpha,\, a_1/(4a_2),\, a_1\Delta_0/(2a_4\tau^2 M_f^2)\big\}, \qquad \bar\beta' = \min\{\bar\beta,\, \tau\bar\alpha'\}.
\]
Then $\delta(X_+) \le (1 - a_1\alpha/2)\delta + a_4\tau^2\alpha^2 M_f^2$. For $\delta\le\Delta_0$, the right-hand side is at most $(1-a_1\alpha/2)\Delta_0 + a_4\tau^2\alpha^2 M_f^2 \le \Delta_0$, where the last step uses $a_4\tau^2\alpha M_f^2 \le a_1\Delta_0/2$.
\end{proof}

By induction, if $X_0\in\manifold$ with $\delta(X_0)\le\Delta_0$ and the step sizes satisfy $0<\alpha_k\le\bar\alpha'$, $0<\beta_k\le\bar\beta'$, $\beta_k\le\tau\alpha_k$ for all $k$, then the iterates generated by Algorithm~\ref{alg:GOTD} remain in the tube $\kanifold$, i.e., $\delta(X_k)\le\Delta_0<\Delta$ for all $k\ge 0$.

\subsection{Lyapunov function and iteration complexity}\label{subsec:lyap_complexity}
We consider the following Lyapunov function to treat the feasibility and the optimality measures,
\begin{equation}\label{eq:Lyap_def}
    \lanifold_\lambda(X):=f(X)+\lambda\,\delta(X),\ \text{with}\ \lambda >0\ \text{as a balance factor}.
\end{equation}
Combining the estimates in Propositions~\ref{prop:delta_descent} and~\ref{prop:f_descent} points to the complexity analysis.

\begin{theorem}\label{thm:lyap_complexity}
Under Assumptions~\ref{assu:h}--\ref{assu:angle}, let $0<\Delta_0<\Delta$ and $X_0\in\manifold$ with $\delta(X_0)\le\Delta_0$. Let $\lambda = 2M_fC_n/a_1$. There exist $\bar\alpha''$, $\bar\beta''$, $\tau''$, $c_\delta$, $c_t>0$ such that for constant step sizes $\alpha\le\bar\alpha''$, $\beta\le\bar\beta''$, $\beta\le\tau''\alpha$, the iterates of Algorithm~\ref{alg:GOTD} satisfy $X_k\in\kanifold\cap\manifold$ and $\lanifold_\lambda(X_{k+1})
    \le
    \lanifold_\lambda(X_k)
    -c_\delta\,\alpha\,\delta(X_k)
    -c_t\,\beta\,\|\Gf(X_k)\|^2$, for all $k\ge 0$.
Moreover, for all $K\ge 1$, we have
\begin{equation}\label{eq:complexity_bounds}
    \min_{0\le k\le K-1}\delta(X_k)
    \le\frac{\lanifold_\lambda(X_0)-f_{\inf}}{c_\delta\alpha K},
    \qquad
    \min_{0\le k\le K-1}\|\Gf(X_k)\|^2
    \le\frac{\lanifold_\lambda(X_0)-f_{\inf}}{c_t\beta K},
\end{equation}
and $\delta(X_k)\to 0$, $\|\Gf(X_k)\|\to 0$ as $k\to\infty$.
\end{theorem}

\begin{proof}
Set $\bar\alpha'' = \min\{\bar\alpha',\, M_fC_n/(4(\lambda a_2 + \frac{L_f}{2}C_n^2\Delta))\}$, $\bar\beta'' = \min\{\bar\beta',\, 1/(2L_f + 4\lambda a_4)\}$, and $\tau'' = \min\{\tau,\, M_fC_n/(4\lambda^2 a_3^2\Delta)\}$. Lemma~\ref{lem:invariant_tube} reveals that $\delta(X_k)\le\Delta_0$ for all $k$. Fix $k$ and write $X:=X_k$, $X_+:=X_{k+1}$, $\delta:=\delta(X)$. Substituting $\|\Gh\|\le C_n\delta$ (Lemma~\ref{lem:Gn_delta}) into Proposition~\ref{prop:f_descent} gives $f(X_+)-f(X)\le-\beta\big(1-\tfrac{L_f}{2}\beta\big)\|\Gf\|^2+\alpha\,M_f C_n\,\delta +\tfrac{L_f}{2}\alpha^2 C_n^2\,\delta^2$.
Adding $\lambda$ times the feasibility estimate $\delta(X_+)-\delta \le -a_1\alpha\delta + a_2\alpha^2\delta + a_3\beta\delta\|\Gf\| + a_4\beta^2\|\Gf\|^2$ from~\eqref{eq:delta_descent} and relaxing $\delta^2$ by $\Delta\delta$, we have
\begin{equation}\label{eq:Lyap_raw}
\begin{aligned}
    \lanifold_\lambda(X_+)-\lanifold_\lambda(X)
    \le\;&
    -\beta\big(1-\tfrac{L_f}{2}\beta - \lambda a_4\beta\big)\|\Gf\|^2
    +\lambda\,a_3\,\beta\,\delta\,\|\Gf\|
    \\
    &-\big(\lambda\,a_1 - M_f C_n\big)\alpha\,\delta
    +\big(\lambda a_2 + \tfrac{L_f}{2}C_n^2\Delta\big)\alpha^2\delta.
\end{aligned}
\end{equation}
Applying Young's inequality $\lambda a_3\beta\delta\|\Gf\| \le \frac{\beta}{4}\|\Gf\|^2 + \lambda^2 a_3^2\Delta\beta\delta$ and $\beta\le\tau''\alpha$ lead to
\begin{equation}\label{eq:Lyap_3term}
\begin{aligned}
    \lanifold_\lambda(X_+)-\lanifold_\lambda(X)
    \le\;&
    -\beta\big(\tfrac{3}{4}-\tfrac{L_f}{2}\beta - \lambda a_4\beta\big)\|\Gf\|^2
    \\
    &-\big(\lambda a_1 - M_fC_n - \lambda^2 a_3^2\Delta\tau''\big)\alpha\delta
    +\big(\lambda a_2 + \tfrac{L_f}{2}C_n^2\Delta\big)\alpha^2\delta.
\end{aligned}
\end{equation}
For the $\|\Gf\|^2$ coefficient, $\bar\beta'' \le 1/(2L_f+4\lambda a_4)$ ensures $L_f\beta/2 + \lambda a_4\beta \le 1/4$, and thus $\tfrac{3}{4}-\tfrac{L_f}{2}\beta - \lambda a_4\beta \ge \tfrac{1}{2}$. For the $\alpha\delta$ coefficient, $\lambda = 2M_fC_n/a_1$ gives $\lambda a_1 - M_fC_n = M_fC_n$, and $\tau'' \le M_fC_n/(4\lambda^2 a_3^2\Delta)$ gives $\lambda^2 a_3^2\Delta\tau'' \le M_fC_n/4$, hence $\lambda a_1 - M_fC_n - \lambda^2 a_3^2\Delta\tau'' \ge \tfrac{3}{4}M_fC_n$. For the $\alpha^2\delta$ term, the choice of $\bar\alpha''$ ensures $(\lambda a_2 + \tfrac{L_f}{2}C_n^2\Delta)\alpha \le M_fC_n/4$. Assembling the estimates, we obtain
$\lanifold_\lambda(X_+)-\lanifold_\lambda(X) \le -c_\delta\alpha\delta - c_t\beta\|\Gf\|^2$ with $c_\delta = M_fC_n/2$ and $c_t = 1/2$. Telescoping from $k=0$ to $K-1$ leads to $\sum_{k=0}^{K-1}\big(c_\delta\alpha\,\delta(X_k)+c_t\beta\|\Gf(X_k)\|^2\big)
\le \lanifold_\lambda(X_0)-f_{\inf} < \infty$,
from which~\eqref{eq:complexity_bounds} and $\delta(X_k)\to 0$, $\|\Gf(X_k)\|\to 0$ follow.
\end{proof}

We present the first-order stationarity condition of problem~\eqref{eq:HM_opt} as follows, which is necessary for the local optimality~\cite[Theorem~6.12]{rockafellar2009variationalanalysis}.
\begin{definition}\label{def:stationary}
    A point $\bar X\in\hanifold\cap\manifold$ is called \emph{stationary} for problem~\eqref{eq:HM_opt} if $\innerp{\nabla f(\bar X),\eta}\ge 0$ for all $\eta\in\tangent_{\hanifold\cap\manifold}(\bar X)$, i.e., $-\nabla f(\bar X)\in\normal_{\hanifold\cap\manifold}(\bar X)$, or equivalently, the projected negative gradient vanishes, i.e., $\projection_{\tangent_{\hanifold\cap\manifold}(\bar X)}(-\nabla f(\bar X))=0$.
\end{definition}
Therefore, the quantity $\|\Gf(X)\|=\|\projection_{\tangent_{\hanifoldX\cap\manifold}(X)}(-\nabla f(X))\|$ serves as a suitable evaluation of the stationarity. In this view, \eqref{eq:complexity_bounds} in Theorem~\ref{thm:lyap_complexity} indeed delivers an $\complexity(1/\sqrt{K})$ convergence rate for both the feasibility and the optimality measures.

We conclude by examining the first-order stationarity of the accumulation points.
\begin{corollary}\label{cor:stationary}
Under conditions of Theorem~\ref{thm:lyap_complexity}, every accumulation point $\bar X\in\manifold$ of $\{X_k\}$ satisfies $h(\bar X)=0$. If, additionally, $\diff h_X|_{\tangent_\manifold(X)}$ has constant rank near $\bar X$ on $\manifold$, then $\bar X$ is a first-order stationary point of~\eqref{eq:HM_opt}.
\end{corollary}
\begin{proof}
By Theorem~\ref{thm:lyap_complexity}, $\delta(X_k)\to 0$ and $\|\Gf(X_k)\|\to 0$. The continuity of $\delta$ and $X_{k_j}\to\bar X$ give $\delta(\bar X) = \lim_{j}\delta(X_{k_j}) = 0$, indicating that $\bar X\in\hanifold$.

For stationarity, denote $L_X:=\diff h_X|_{\tangent_\manifold(X)}:\tangent_\manifold(X)\to\mbR^q$. As shown in the proof of Proposition~\ref{pro:proj_SX}, the adjoint $L_X^*:\mbR^q\to\tangent_\manifold(X)$ coincides with $\Phi_\manifold(X) = \projection_{\tangent_\manifold(X)}\circ\diff h_X^*$, thereby $\diff h_X\circ\Phi_\manifold(X) = L_X L_X^*\in\mbR^{q\times q}$. Since $\mathrm{rank}(L_X L_X^*) = \mathrm{rank}(L_X)$, the constant rank condition on $L_X$ is equivalent to that on $L_X L_X^*$, and thus the pseudoinverse of $L_X L_X^*$ is continuous. By the formula~\eqref{eq:proj_SX} in Proposition~\ref{pro:proj_SX}, the map $X\mapsto\projection_{S(X)}$ is then continuous near $\bar X$. Passing to the limit along $X_{k_j}\to\bar X$ in $\Gf(X_{k_j})=\projection_{S(X_{k_j})}(-\nabla f(X_{k_j}))\to 0$ gives $\projection_{S(\bar X)}(-\nabla f(\bar X))=0$.
\end{proof}

\begin{remark}\label{rem:RCRCQ}
The constant rank condition on $\diff h_X|_{\tangent_\manifold(X)}$ coincides with the \emph{relaxed constant rank constraint qualification} (RCRCQ) introduced in~\cite{andreani2024CQ_RieALM} for Riemannian optimization problems with equality constraints. This, in views of Theorem~\ref{thm:IT_implies_clean}, amounts to requiring that the slices $\hanifoldX\cap\manifold$ have constant dimension as $X$ varies near $\bar X$ on $\manifold$. Such a property is satisfied by all the instances outlined in Table~\ref{tab:tangentsets}.
\end{remark}

\section{Numerical experiments}\label{sec:experiments}
In this section, we evaluate the proposed GOTD (Algorithm~\ref{alg:GOTD}) on three applications within the scope of formulation~\eqref{eq:HM_opt}, each constrained to an intersection of two manifolds. We adopt the Riemannian augmented Lagrangian method (RALM)~\cite[Algorithm~1]{liu2020simple} as a baseline, which can handle Riemannian problems with equality and inequality constraints; we run RALM with the publicly available implementation.\footnote{\url{https://github.com/losangle/Optimization-on-manifolds-with-extra-constraints}} In addition, each experiment also incorporates some other task-specific methods for comparison, where we note that the Riemannian trust-region method on manifolds is run with \texttt{Manopt}'s default settings. The experiments are produced on a workstation that consists of two Intel(R) Xeon(R) Gold 6330 CPUs (at $2.00$GHz$\times28$, $42$M Cache), 512GB RAM. All the experiments are carried out in MATLAB (Release {9.7.0}) on the CPUs, drawing on the \texttt{Manopt} toolbox~\cite{boumal2014manopt}. The codes of the proposed method are available at \href{https://github.com/UCAS-YanYang}{https://github.com/UCAS-YanYang}.

\subsection{Low-rank approximation of spherical data}\label{subsec:sphere} Finding a low-rank approximation of normalized data plays a crucial role in various applications. Given $A\in\oblique(m,n):=\{X\in\mbR^{m\times n}\mid \ddiag(XX^\top)-\mathbf{1}=\mathbf{0}\}$, where rows encode data points with unit length, Chu et al.~\cite{chu2005lowrankoblique} formulated the approximation task as follows,
\begin{equation}\label{eq:sphere_problem}
    \begin{aligned}
        \min_{X\in\mbR^{m\times n}}\ \ &\frac{1}{2} \norm{\projection_{\Omega}(X-A)}^2\\
        \mathrm{s.\,t.}\ \ \ \ \,&X\in\oblique(m,n)\cap\fixedrank,
    \end{aligned}
\end{equation}
where $\Omega\subseteq\{1,2,\ldots,m\}\times\{1,2,\ldots,n\}$ represents observed entries and $\projection_\Omega$ defines the sampling operator: $\projection_{\Omega}(X)(i,j) = X(i,j)$ if $(i,j)\in\Omega$, otherwise $\projection_{\Omega}(X)(i,j) = 0$. 

Apart from RALM, we also take into account the method proposed in~\cite{yang2025spacedecouple}, which parameterizes $\oblique(m,n)\cap\fixedrank$ by a smooth manifold denoted by $\manifold_h$. Subsequently, \texttt{Manopt}'s Riemannian gradient descent and Riemannian trust-region methods are invoked on $\manifold_h$ for comparison, denoted by $\manifold_h$-RGD and $\manifold_h$-RTR, respectively.

Following the test of~\cite{yang2025spacedecouple}, we generate a ground truth $A=\projection_{\oblique(m,n)}(U^*\varSigma^*(V^*)^\top)$, where $U^*\in\stiefel(m,r)$ and $V^*\in\stiefel(n,r)$ are obtained by sampling entries from the standard normal distribution $\mathcal{N}(0,1)$ and taking the Q-factors of the QR factorizations, and $\varSigma^*\in\mbR^{r\times r}$ is a diagonal matrix with entries sampled from the uniform distribution on $(0,1)$. The oversampling factor is defined by $\mathrm{OS}:=|\Omega|/(r(m+n-r))$. All the methods share a common initial point $X_0=H_0V_0^\top$, where $V_0$ is generated in the same way as $V^*$ and $H_0=\projection_{\oblique(m,r^*)}(\tilde H_0)$, where entries of $\tilde H_0$ are sampled from $\mathcal{N}(0,1)$. The termination rules are set for respective methods: $\max\{\|\Gh(X_k)\|,\|\Gf(X_k)\|\}\le 10^{-10}$ for GOTD, \emph{subproblem accuracy} of RALM achieves the tolerance $10^{-8}$ (see~\cite[Algorithm~1]{liu2020simple}), the norm of Riemannian gradient on $\manifold_h$ achieves $10^{-10}$ for $\manifold_h$-RGD and $10^{-13}$ for $\manifold_h$-RTR. Reconstruction quality is measured by the relative test error $\|\projection_\varGamma(X-A)\|_\frob/\|\projection_\varGamma(A)\|_\frob$ on an independent test set $\varGamma$ with $|\varGamma|=|\Omega|$. All reported numbers are averaged over five random seeds.

We test with the dimension $(m,n)=(5000,6000)$, the oversampling factors $\mathrm{OS}\in\{6,7,8,9,10\}$, and the rank parameters $r\in\{8,9,10\}$. For GOTD, we adopt constant step sizes throughout: $\alpha=1$ in all experiments and $\beta$ tuned over the grid $\{1,5,10,20,30,40,50\}$. The same tuning strategy applies to the constant step size in $\manifold_h$-RGD. Table~\ref{tab:sphere_gotd_sweep} reports the performance of GOTD across different $(\mathrm{OS},r)$, with the true data matrix successfully recovered in every configuration.

\begin{table}[htbp]
\centering
\caption{Performance of GOTD on spherical data fitting problem with different oversampling factors and rank parameters, each entry averaged over five seeds.}
\label{tab:sphere_gotd_sweep}
\setlength{\extrarowheight}{0.5ex}
\footnotesize
\begin{tabular*}{\textwidth}{@{\extracolsep{\fill}}ccccccc}
\toprule
\multicolumn{1}{l}{\multirow{2}{*}{$\mathrm{OS}$}} & \multicolumn{2}{c}{$r=8$} & \multicolumn{2}{c}{$r=9$} & \multicolumn{2}{c}{$r=10$}\\
\cmidrule(l{2pt}r{2pt}){2-3}\cmidrule(l{2pt}r{2pt}){4-5}\cmidrule(l{2pt}r{2pt}){6-7}
 & Test err. & Time & Test err. & Time & Test err. & Time \\
\midrule
$6$  & $\expnumber{1.51}{-11}$ & $\phantom{1}5.41$  & $\expnumber{9.38}{-12}$ & $\phantom{1}3.17$  & $\expnumber{8.40}{-12}$ & $\phantom{1}5.50$  \\
$7$  & $\expnumber{9.21}{-12}$ & $\phantom{1}2.71$  & $\expnumber{8.05}{-12}$ & $\phantom{1}2.72$  & $\expnumber{8.79}{-12}$ & $\phantom{1}7.06$  \\
$8$  & $\expnumber{8.80}{-12}$ & $\phantom{1}4.45$  & $\expnumber{1.53}{-12}$ & $\phantom{1}2.83$  & $\expnumber{6.76}{-13}$ & $13.56$ \\
$9$  & $\expnumber{1.69}{-11}$ & $\phantom{1}5.69$  & $\expnumber{6.63}{-13}$ & $13.58$ & $\expnumber{6.54}{-12}$ & $12.15$ \\
$10$ & $\expnumber{2.83}{-12}$ & $\phantom{1}2.32$  & $\expnumber{6.84}{-13}$ & $13.89$ & $\expnumber{6.97}{-13}$ & $\phantom{1}9.96$  \\
\bottomrule
\end{tabular*}
\end{table}

Moreover, Figure~\ref{fig:sphere} compares the four methods on the scenario $(\mathrm{OS},r)=(6,10)$, where the feasibility is measured by $\|h(X)\|=\|\ddiag(XX^\top)-\mathbf{1}\|$. RALM and $\manifold_h$-RTR exhibit better iteration complexity, since the former solves a Lagrangian subproblem at each update and the latter exploits Hessian information. In terms of running time, GOTD is more efficient, as each of its iterations invokes only first-order oracles and one retraction on $\fixedrank$, without additional inner subproblems to solve.

\begin{figure}[htbp]
    \centering
    \hspace{3mm}\includegraphics[width=0.9\textwidth]{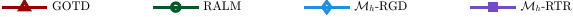}\\[2mm]
    \begin{minipage}{0.32\textwidth}\centering\includegraphics[width=\linewidth]{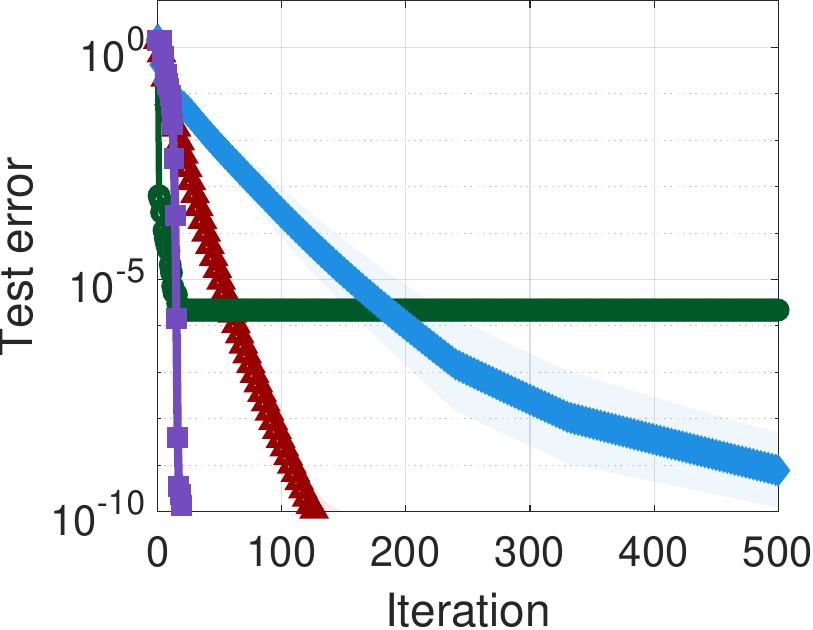}\end{minipage}\hfill
    \begin{minipage}{0.32\textwidth}\centering\includegraphics[width=\linewidth]{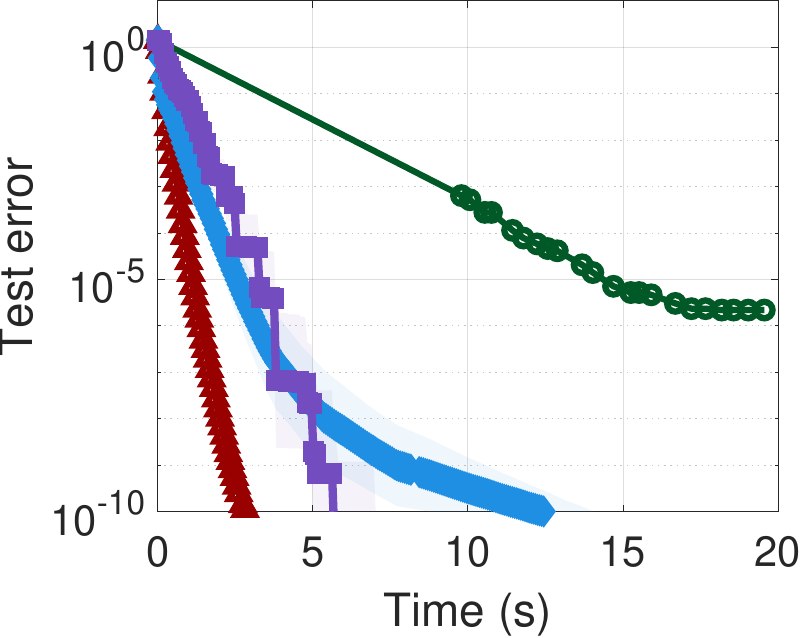}\end{minipage}\hfill
    \begin{minipage}{0.32\textwidth}\centering\includegraphics[width=\linewidth]{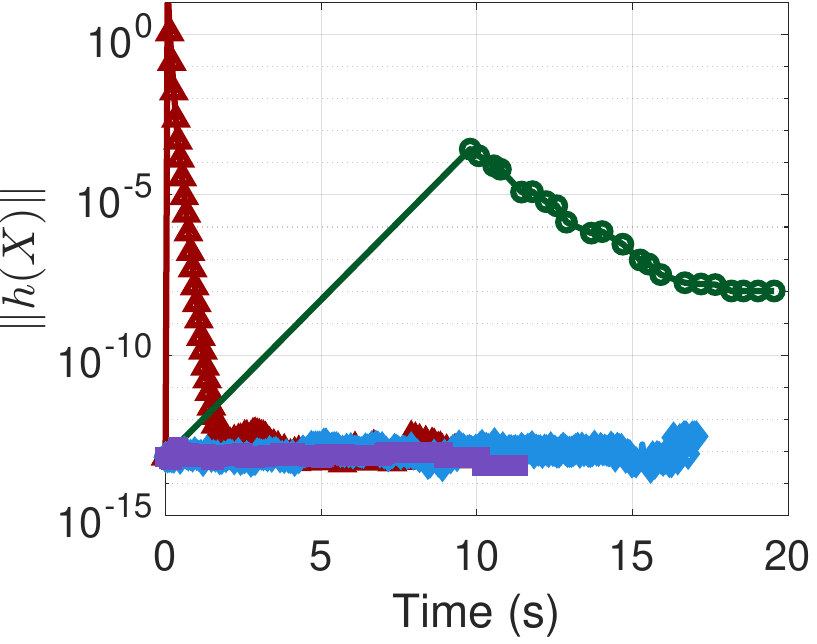}\end{minipage}
    \caption{Low-rank approximation of spherical data with $(m,n)=(5000,6000)$, $r=r^*=10$, and $\mathrm{OS}=6$. Left: test error versus iteration. Middle: test error versus wall-clock time. Right: feasibility violation $\|\ddiag(X_kX_k^\top)-\mathbf{1}\|$ versus wall-clock time.}
    \label{fig:sphere}
\end{figure}

\subsection{Low-rank approximation of hyperbolic embeddings}\label{subsec:hyperbolic}
Hyperbolic embeddings have shown effectiveness in representing hierarchical data, with applications ranging from graph embeddings to natural language processing~\cite{nickel2017poincare,jawanpuria2019lowrankhyperbolic}. To capture the hyperbolic geometry, the \emph{hyperboloid model} is introduced as $\mathbb{H}_n := \{y \in \mbR^{n+1} \mid \innerp{y,y}_J +1 = 0,\ y_1 > 0\}$,
where $y_1$ is the first element of the vector, $J := \Diag(-1, 1, \ldots, 1)$ is the Lorentz signature, and $\innerp{x,y}_J := x^\top J y$ denotes the Lorentzian inner product. As the upper sheet of an $n$-dimensional hyperboloid in $\mbR^{n+1}$, $\mathbb{H}^n$ is a smooth manifold with the distance $d_\mathbb{H}(x,y) := \operatorname{arccosh}(-\innerp{x,y}_J)$.

Given $m$ data points which have been embedded in $\mathbb{H}^n$ as $\bar x_1, \ldots, \bar x_m$, Jawanpuria et al.~\cite{jawanpuria2019lowrankhyperbolic} proposed to seek $x_1, \ldots, x_m \in \mathbb{H}^n$ as the approximations, which share a latent low-dimensional subspace---that is, the matrix $X = [x_1, \ldots, x_m]$ is low-rank. Stacking the constraints $x_i \in \mathbb{H}^n$ translates to $X$ lying on the \emph{matrix hyperboloid},
\begin{equation}\label{eq:matrixhyperboloid}
    \matrixhyperboloid := \{Y \in \mbR^{(n+1) \times m} \mid \ddiag(Y^\top J Y) + \mathbf{1} = \mathbf{0},\, Y_{1:}>\mathbf{0}\},
\end{equation}
where $Y_{1:}$ extracts the first row of $Y$ and the ``$>$'' is understood component-wise. Letting $r \ll \min\{m,n\}$ denote the rank parameter, finding a low-rank approximation of hyperbolic embeddings can be formulated as follows,
\begin{equation}\label{eq:hyperbolic_problem}
    \begin{aligned}
        \min_{X\in\mbR^{(n+1)\times m}}\ \ &f(X)=\sum_{i=1}^m \operatorname{arccosh}\!\big(-\innerp{x_i,\bar x_i}_J\big)^2\\
        \mathrm{s.\,t.}\ \ \ \ \ \ &X\in\matrixhyperboloid\cap \fixedrankplus
    \end{aligned}
\end{equation}
where $x_i$ denotes the $i$-th column of $X$, and the objective is the sum of squared distances between $x_i$ and $\bar x_i$ for $i=1,\ldots,m$.

\begin{table}[t]
\centering
\caption{Results of the four methods on low-rank hyperbolic embeddings on the WordNet mammals subtree. The $f_0$ and $f$ are the initial and returned objective value, respectively. For reference, the mean average precision computed through the original embedding $\bar x_1, \ldots, \bar x_m$ is $0.9385$.}
\label{tab:hyperbolic}
\setlength{\extrarowheight}{0.6ex}
\footnotesize
\begin{tabular*}{\textwidth}{@{\extracolsep{\fill}}clrrcr}
\toprule
{Rank} & {Algorithm} & {$f/f_0$} & {Time (s)} & {$\|h(X)\|$} & {Precision} \\
\midrule
\multirow{4}{*}{$r=5$}
  & GOTD              & \textbf{0.434} & \phantom{0}80.8          & $\expnumber{6.81}{-8\phantom{1}}$  & 0.8869 \\
  & RALM              & 0.994          & $699.2$                  & $\expnumber{1.82}{-6\phantom{1}}$  & 0.6637 \\
  & $\manifold_p$-RGD & 0.439          & \phantom{0}92.9          & $\expnumber{1.89}{-12}$            & 0.8881 \\
  & $\manifold_p$-RTR & 0.435          & \phantom{0}78.3          & $\expnumber{4.14}{-12}$            & \textbf{0.8883} \\
\midrule
\multirow{4}{*}{$r=10$}
  & GOTD              & \textbf{0.394}          & \phantom{0}65.9          & $\expnumber{5.01}{-8\phantom{1}}$  & \textbf{0.9039} \\
  & RALM              & 0.949          & $712.9$                  & $\expnumber{3.19}{-6\phantom{1}}$  & 0.8117 \\
  & $\manifold_p$-RGD & 0.402          & \phantom{0}99.7          & $\expnumber{3.38}{-12}$            & 0.8961 \\
  & $\manifold_p$-RTR & 0.396 & \phantom{0}82.5          & $\expnumber{4.14}{-12}$            & 0.9024 \\
\midrule
\multirow{4}{*}{$r=20$}
  & GOTD              & \textbf{0.369} & \phantom{0}98.8          & $\expnumber{1.64}{-7\phantom{1}}$  & \textbf{0.9003} \\
  & RALM              & 0.877          & $653.9$                  & $\expnumber{3.55}{-6\phantom{1}}$  & 0.8463 \\
  & $\manifold_p$-RGD & 0.380          & 100.3                    & $\expnumber{3.94}{-12}$            & 0.8989 \\
  & $\manifold_p$-RTR & 0.376          & \phantom{0}73.0          & $\expnumber{4.89}{-12}$            & 0.8981 \\
\bottomrule
\end{tabular*}
\end{table}

The existing approach for~\eqref{eq:hyperbolic_problem}, developed by Jawanpuria et al.~\cite{jawanpuria2019lowrankhyperbolic}, considers the product manifold $\manifold_p:=\stiefel(n,r)\times \mathbb{H}^{m}_{r}$, together with a smooth map defined by
\begin{equation*}\label{eq:jaw-phi}
   \phi:\manifold_p\to\mbR^{(n+1)\times m}:\ \phi(U, Z) \;=\; \begin{bmatrix} 1 & \mathbf{0}_r^\top \\ \mathbf{0}_n & U \end{bmatrix} Z.
\end{equation*}
Since $\phi(\manifold_p)\subsetneq\matrixhyperboloid\cap\fixedrankplus$, it produces a surrogate problem $\min_{Z\in\manifold_p} f(\phi(Z))$, thereby allowing the employment of Riemannian gradient descent (RGD) and Riemannian trust-region (RTR) methods on $\manifold_p$. Our method GOTD, in turn, tackles the coupled constraint $X\in\matrixhyperboloid\cap\fixedrankplus$ in the original space $\mbR^{(n+1)\times m}$; the computation of $\Gf(X)$ is referred to Appendix~\ref{app:proj-hyperboloid}. We also include RALM for comparison, interpreting $X\in\matrixhyperboloid$ as the equality constraint $h(X)=\ddiag(X^\top J X)+\mathbf{1}=\mathbf{0}$.

We evaluate the methods on the mammal subtree of the WordNet lexical database~\cite{miller1995wordnet}, with $(n,m)=(300,1170)$ and rank parameters $r\in\{5,10,20\}$. The nodes are nouns and the edges encode the ``is-a'' relationship; for instance, the edge between ``squirrel'' and ``rodent'' indicates that a squirrel is a rodent. The original embeddings $\bar x_1,\ldots,\bar x_m$ are obtained from the implementation~\cite{nickel2017poincare}. The quality of the approximations is measured by the \emph{mean average precision}: given the true edge set $E$, for each pair $(u,v)\in E$, distance $d_\mathbb{H}(x_u,x_v)$ is ranked among the distances $\{d_\mathbb{H}(x_u,x_w)\mid (u,w)\notin\mathcal{E}\}$, and the resulting precision over all true edges is averaged.

Let $\bar X=[\bar x_1, \ldots, \bar x_m]$, and collect the top-$r$ left singular vectors of the submatrix $\bar X^\prime:=\bar X(2{:}n{+}1,:)$ to obtain $U_r\in\stiefel(n,r)$. We denote $\bar Z^\prime := U_r^\top \bar X^\prime$ and write $z_i^\prime$ for its $i$-th column. Augmenting each column as $\bar z_i := [\sqrt{1+\|z_i^\prime\|^2},\, z_i^{\prime\top}]^\top \in \mathbb{H}_r$ for $i=1,\ldots,m$ and assembling them into $\bar Z = [\bar z_1, \ldots, \bar z_m]$, we initialize $\manifold_p$-RGD and $\manifold_p$-RTR at $(U_r, \bar Z)$, and initialize GOTD and RALM at $X_0 = \phi(U_r, \bar Z)$. After conducting initial tests, we use $\alpha=1$ and $\beta$ tuned over $\{0.1,0.15,0.2,0.5,0.75,1\}$ for GOTD; the same tuning strategy applies to the constant step size in $\manifold_p$-RGD.

Table~\ref{tab:hyperbolic} shows that the feasibility measure $\|h\|$ reaches the order of $10^{-7}$ for GOTD, which is acceptable, and stays at machine precision for $\manifold_p$-RGD and $\manifold_p$-RTR, since the methods are feasible in the sense that $\phi(\manifold_p)\subset\matrixhyperboloid$. For a fair comparison, we post-process the output of GOTD by projecting it onto $\matrixhyperboloid\cap\fixedrankplus$ via alternating projections, yielding $X_{\mathrm{gotd}}$, on which we report the ``$f/f_0$'' and the precision. Across the three ranks, GOTD attains the lowest cost $f$, together with the highest mean average precision at $r=10$ and $r=20$. A plausible reason for the enhanced performance is the observation $X_{\mathrm{gotd}}\notin\phi(\manifold_p)$, which indicates that GOTD searches over a larger set. Moreover, the running time of GOTD is competitive with the Riemannian algorithms conducted on $\manifold_p$.

\subsection{Compressed modes in physics}\label{subsec:CMP}
In quantum mechanics, spatially localized solutions to Schr\"odinger's equation have attracted interest recently~\cite{ozolins2013compressed}. Such solutions, known as \emph{compressed modes}, translate into the sparsity of the discretization matrix. Specifically, with $X\in\stiefel(n,p)$ collecting $p$ orthonormal wave functions, the task of finding compressed modes in physics can be formulated as follows,
\begin{equation}\label{eq:CMP_problem}
    \begin{aligned}
        \min_{X\in\mbR^{n\times p}}\ \ &\trace(X^\top A X)\\
        \mathrm{s.\,t.}\ \ \ \ \,&X\in\stiefel(n,p)\cap\canifold_s
        \\[-0.5mm]
    \end{aligned}
\end{equation}
where $\sparseset=\{Y\in\mbR^{n\times p}\mid\|Y\|_0= s\}$, $A\in\mbR^{n\times n}$ is a symmetric matrix discretizing the Hamiltonian, and $s>0$ is a parameter depicting the sparsity. Hence~\eqref{eq:CMP_problem} is an instance of~\eqref{eq:HM_opt} with $(\hanifold,\manifold)=(\stiefel(n,p),\sparseset)$, to which GOTD can be applied.

The existing approaches turn to the relaxation by adding an $\ell_1$ regularization of the wave functions~\cite{ozolins2013compressed,chen2020ManPG}, yielding
\begin{equation}\label{eq:CMP_relaxed}
    \begin{aligned}
        \min_{X\in\mbR^{n\times p}}\ \ &\trace(X^\top A X) + \mu\|X\|_1\\
        \mathrm{s.\,t.}\ \ \ \ \,&X\in\stiefel(n,p).
        \\[-1mm]
    \end{aligned}
\end{equation} 
In the comparison, we implement GOTD and RALM for~\eqref{eq:CMP_problem}, where RALM treats the constraint $X\in\stiefel(n,p)$ as the equality constraint $h(X)=X^\top X-I_p=0$. Moreover, the proximal gradient methods ManPG~\cite[Algorithm~2]{chen2020ManPG} and SLPG~\cite[Algorithm~3]{liu2024SLPG} are applied to~\eqref{eq:CMP_relaxed}, which are able to tackle nonsmooth objectives.

Following the setting of~\cite{chen2020ManPG}, we take the 1D free-electron Hamiltonian $-\tfrac12\partial_x^2$ on the interval $[0, L]$ with $L = 50$, discretized on a uniform grid of $n = 256$ interior points, and seek $p = 15$ orthonormal wave functions. The relaxation parameter for the two $\ell_1$-penalty methods is set to $\mu = 1/30$, taken from the publicly available implementation of ManPG.\footnote{\url{https://github.com/chenshixiang/ManPG/tree/master}} Sparsity is measured by the ratio of the non-zero entry, $\rho(X) := (np - \|X\|_0)/(np)$. To report the sparsity of $X$, we treat an entry as zero whenever $|X_{ij}| < 10^{-8}$ for ManPG and SLPG; the iterates of GOTD and RALM are exactly on $\canifold_s$ at every step, and thus the threshold has no bearing on either method. The initial points for ManPG and SLPG are obtained by running a Riemannian sub-gradient descent on problem~\eqref{eq:CMP_relaxed}, while GOTD and RALM additionally project the points onto $\canifold_s$ as the warm start. The step size for ManPG and SLPG is $1/(2\lambda_{\max}(A)) = L^2/4n^2$ with $\lambda_{\max}(A)$ denoting the largest eigenvalue of $A$. Similarly, we take $\alpha=1$ and $\beta=L^2/4n^2$ for GOTD. The methods ManPG and SLPG are terminated when the iterate satisfies $|F(X_k) - F(X_{k+1})| \le 10^{-7}$ with $F(X):=\trace(X^\top A X) + \mu\|X\|_1$, and GOTD is terminated when $\max\{\|\Gh(X_k)\|,\|\Gf(X_k)\|\}\le 10^{-10}$.

Figure~\ref{fig:CMP} reports the cost $\trace(X^\top AX)$, the sparsity $\rho$, and the feasibility measure $\|h(X)\|=\|X^\top X-I_p\|$ of the iterates. Specifically, 
ManPG and SLPG rapidly sparsify the iterates while maintaining the cost value, with the sparsity steady at $\rho \approx 0.6$; this motivates us to configure the manifold $\canifold_s$ with $s=\rho\times np$ and $\rho\in\{0.6,0.7\}$ in the formulation~\eqref{eq:CMP_problem}. Consequently, the curves reveal that GOTD attains lower cost value under both levels of sparsity, exhibiting the effectiveness and efficiency of the proposed method. In addition, RALM, also targeting the formulation~\eqref{eq:CMP_problem}, approximately matches the cost value returned by GOTD at the same sparsity, but accrues a slightly higher feasibility violation.

\begin{figure}[htbp]
    \centering
    \includegraphics[width=0.95\textwidth]{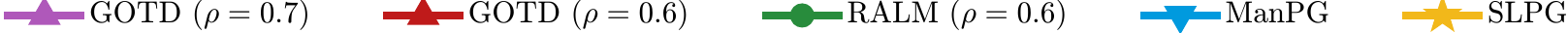}\\[2mm]
    \begin{minipage}{0.32\textwidth}\centering\includegraphics[width=\linewidth]{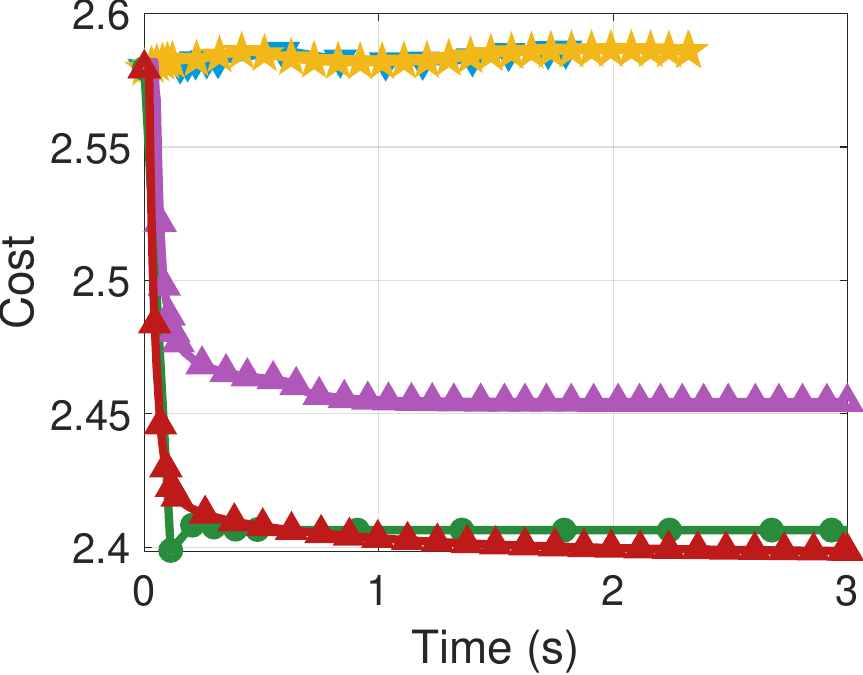}\end{minipage}\hfill
    \begin{minipage}{0.32\textwidth}\centering\includegraphics[width=\linewidth]{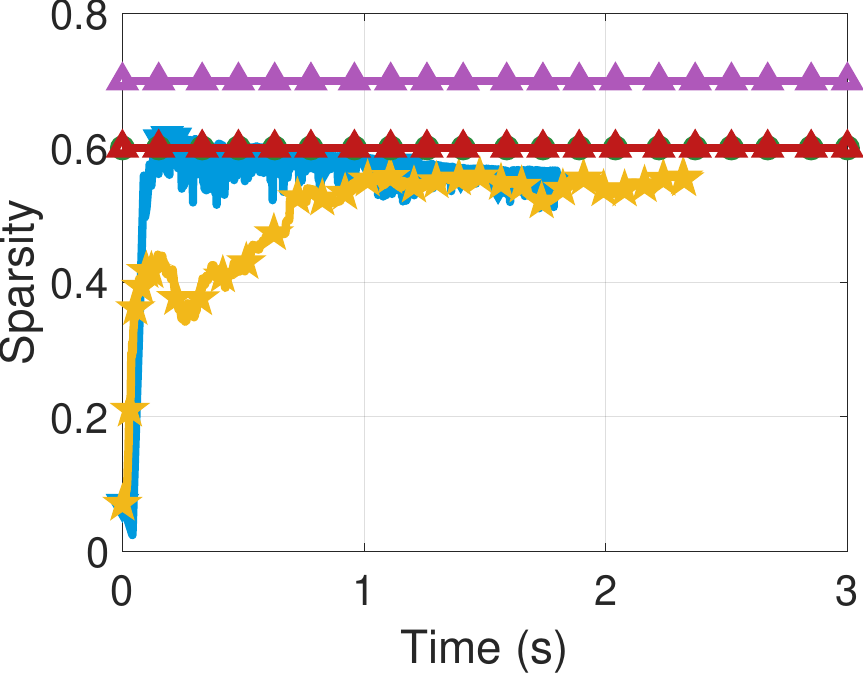}\end{minipage}\hfill
    \begin{minipage}{0.32\textwidth}\centering\includegraphics[width=\linewidth]{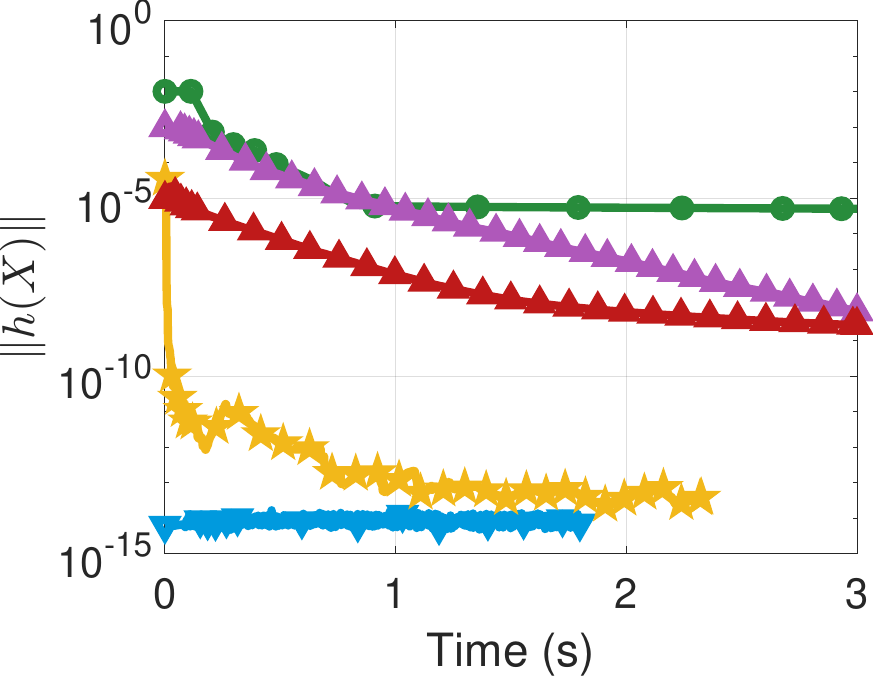}\end{minipage}
    \caption{Comparison on finding compressed modes with $(n,p)=(256,15)$.}
    \label{fig:CMP}
\end{figure}

\section{Conclusions and perspectives}\label{sec:conclusion}
In this work, we propose GOTD, a geometric method for optimization over the intersection of two manifolds $\hanifold\cap\manifold$, under intrinsic transversality. The main principle is to decompose the update into two orthogonal directions tangent to $\manifold$: the projected Gauss--Newton direction improving the feasibility and the projected anti-gradient direction decreasing the objective. Importantly, constructions of the two components are promoted by two equivalent characterizations of intrinsic transversality in Theorem~\ref{thm:PB-implies-IT} and Theorem~\ref{thm:IT_implies_clean}, respectively. Convergence analysis and numerical experiments validate the effectiveness and efficiency of GOTD.

The developed theory suggests several avenues worth exploration. First, given the first-order intersection rule~\eqref{eq:tangent_decom} enlightening our work, it is reasonable to consider the second-order counterpart $\tangent^2_{\hanifold\cap\manifold}(X;\eta)=\tangent^2_\hanifold(X;\eta)\cap\tangent^2_\manifold(X;\eta)$ identified in~\cite{yang20252ndVariety} to devise a second-order extension of the proposed GOTD, where $\tangenttwo$ denotes the second-order tangent set. In addition, adapting Algorithm~\ref{alg:GOTD} to accommodate stochastic oracles is of interest for large-scale scientific computation where only noisy estimates of $\nabla f$ and $h$ are accessible; some relevant techniques can be found in~\cite{schechtman2023ODCGM,ablin2024infeasible}. Moreover, extending GOTD to handle multiple intersecting constraints $X\in\manifold\cap\hanifold_1\cap\cdots\cap\hanifold_N$ also appears as a potential direction, which can borrow some ideas from~\cite{lewis2009localaveraged}.

\section*{Acknowledgments}
The authors would like to thank P.-A. Absil for his comments and insightful suggestions on this work.

\newpage

\appendix

\section{Projection onto $\tangent_{\matrixhyperboloid\cap\lowrank}(X)$}\label{app:proj-hyperboloid} Given a rank parameter $s$ with $1\le s\le \min\{(n+1),m\}$, we then derive an efficient procedure for computing the projection onto the space $\tangent_{\matrixhyperboloid\cap\lowrank}(X)$, which is invoked by the experiment in section~\ref{subsec:hyperbolic}. Let $J=\Diag(-1,1,\ldots,1)\in\mbR^{(n+1)\times(n+1)}$, and write $X_i$ for the $i$-th column of a matrix $X$. The matrix hyperboloid $\matrixhyperboloid$ defined in~\eqref{eq:matrixhyperboloid} is the level set of $h:\mbR^{(n+1)\times m}\to\mbR^m$ with $h_j(X)=X_j^\top JX_j+1$ and the differentials $\diff h_X(Z)=2\ddiag(X^\top JZ)$ and $\diff h_X^*(\lambda)=2JX\Diag(\lambda)$.
Moreover, the normal and tangent spaces are $\normal_{\matrixhyperboloid}(X)=\{JX\Diag(\lambda)\mid \lambda\in\mbR^m\}$ and $\tangent_{\matrixhyperboloid}(X)=\{Z\in\mbR^{(n+1)\times m}\mid X_j^\top JZ_j=0,\;j=1,\ldots,m\}$, respectively. Regarding the geometry of $\lowrank$, we consider the singular value decomposition $X=U\Sigma V^\top$ with $U\in\stiefel(n+1,s)$ and $V\in\stiefel(m,s)$ to give the characterization $\tangent_{\lowrank}(X)
=\{UU^\top Z+ZVV^\top-UU^\top ZVV^\top\mid Z\in\mbR^{(n+1)\times m}\}$, and the orthogonal projection $\projection_{\tangent_{\lowrank}(X)}(Z) = UU^\top Z+ZVV^\top-UU^\top ZVV^\top$; see the developments in~\cite{vandereycken2013lowrankcompletion,olikier2026tangentconeproof}.
Building on the geometry of each manifold, we investigate the intersection.
\begin{proposition}\label{prop:hyp_transversal}
At every point $X \in \matrixhyperboloid \cap \lowrank$, the manifolds $\matrixhyperboloid$ and $\lowrank$ intersect transversally, i.e., $\tangent_{\matrixhyperboloid}(X) + \tangent_{\lowrank}(X) = \mbR^{(n+1)\times m}$. Consequently, $\matrixhyperboloid \cap \lowrank$ is a smooth submanifold of dimension $s(n+m-s+1)-m$.
\end{proposition}
\begin{proof}
It suffices to prove that $\normal_{\matrixhyperboloid}(X)\cap\normal_{\lowrank}(X)=\{0\}$ by definition of transversality. To see this, take any $JX\Diag(\lambda)\in\normal_{\matrixhyperboloid}(X)$, and suppose additionally that it belongs to $\normal_{\lowrank}(X)$. Note that for any $C\in\mbR^{m\times m}$, we have $XC\in\tangent_{\lowrank}(X)$, and thus $JX\Diag(\lambda)\in \normal_{\lowrank}(X)$ gives $\innerp{JX\Diag(\lambda),\,XC} = 0$, which forces $X^\top JX\Diag(\lambda)=0$. Taking the diagonal and recalling $\ddiag(X^\top JX)=-\mathbf{1}$ yields $\lambda=0$, thereby $\normal_{\matrixhyperboloid}(X)\cap\normal_{\lowrank}(X)=\{0\}$. The dimension of the intersection follows from $\dime(\lowrank)+\dime(\matrixhyperboloid)-\dime(\mbR^{(n+1)\times m})=s(n+m-s+1)-m$.
\end{proof}

With Proposition~\ref{prop:hyp_transversal} depicting the transversality, we can deliver the computation of the projection onto $\tangent_{\matrixhyperboloid\cap\lowrank}(X)$ via Proposition~\ref{pro:proj_SX}. To exploit the low-rank structure $X=U\varSigma V^\top$, we decompose $JX=UP+Q$ with $P:=U^\top JX\in\mbR^{s\times m}$ and $Q:=(I-UU^\top)JX\in\mbR^{(n+1)\times m}$. Following the notation in Proposition~\ref{pro:proj_SX} that $\Phi_{\lowrank}(X)=\projection_{\tangent_{\lowrank}(X)}\circ\diff h_X^*$, we obtain the following identity,
\begin{equation}\label{eq:Phi_lowrank}
\Phi_{\lowrank}(X)(\lambda)=2\projection_{\tangent_{\lowrank}(X)}\!\big(JX\Diag(\lambda)\big)=2\big[U\,P\Diag(\lambda)+Q\Diag(\lambda)\,VV^\top\big],
\end{equation}
Letting $D_P:=\Diag\!\big(\|P_1\|_2^2,\ldots,\|P_m\|_2^2\big)\in\mbR^{m\times m}$ collect the squared column norms of $P$, a direct computation then gives, for every $\lambda\in\mbR^m$, $\diff h_X\circ\Phi_{\lowrank}(X)[\lambda]=4A\lambda$ with $A:=D_P+(Q^\top Q)\odot(VV^\top)$. Propositions~\ref{pro:proj_SX} and~\ref{prop:hyp_transversal} guarantee that $A$ is invertible. Substituting~\eqref{eq:Phi_lowrank} into~\eqref{eq:proj_SX} with $\xi\in\mbR^{(n+1)\times m}$ and $\eta:=\projection_{\tangent_{\lowrank}(X)}(\xi)$, we obtain
\begin{equation}\label{eq:proj-KKT}
\projection_{\tangent_{\matrixhyperboloid\cap\lowrank}(X)}(\xi) = \eta-U(P\Diag(\lambda))-(Q\Diag(\lambda)V)V^\top,
\end{equation}
where $\lambda\in\mbR^m$ solves the linear system $A\lambda=b$, with the $i$-th entry of $b$ given by $b_i:=(JX_i)^\top\eta_i$. Although $A$ is of size $m\times m$, it need not be assembled in advance. In fact, the action on any $w\in\mbR^m$ admits the factored form $Aw=D_P\,w+\sum_{l=1}^{s}V_l\odot\bigl(Q^\top(Q(V_l\odot w))\bigr)$, where each matrix-vector product involving $Q$ takes $\complexity((n+m)s)$ flops via the SVD of $X$, and thus computing $Aw$ costs $\complexity(s^2(n+m))$. Therefore, we can solve $A\lambda=b$ by the preconditioned conjugate gradient method, with $\Diag(\diag(A))$ as the preconditioner. Taking $\xi=-\nabla f(X)$ for~\eqref{eq:proj-KKT} yields the $\Gf(X)$ in Algorithm~\ref{alg:GOTD}.

\newpage

\bibliographystyle{siamplain}
\bibliography{references}

@article{boumal2014manopt,
  title={Manopt, a {M}atlab toolbox for optimization on manifolds},
  author={Boumal, Nicolas and Mishra, Bamdev and Absil, P.-A. and Sepulchre, Rodolphe},
  journal={The Journal of Machine Learning Research},
  volume={15},
  number={1},
  pages={1455--1459},
  year={2014},
  publisher={JMLR. org}
}

@article{federer1959curvature,
  title={Curvature measures},
  author={Federer, Herbert},
  journal={Transactions of the American Mathematical Society},
  volume={93},
  number={3},
  pages={418--491},
  year={1959}
}

@article{boumal2019global,
  title={Global rates of convergence for nonconvex optimization on manifolds},
  author={Boumal, Nicolas and Absil, Pierre-Antoine and Cartis, Coralia},
  journal={IMA Journal of Numerical Analysis},
  volume={39},
  number={1},
  pages={1--33},
  year={2019},
  publisher={Oxford University Press}
}

@article{neumann1950functional,
  title={Functional Operators},
  author={Neumann, Von},
  journal={The Geometry of Orthogonal Spaces},
  year={1950},
  publisher={Princeton Univ. Press}
}

@book{Hormander1985III,
  author    = {H{\"o}rmander, Lars},
  title     = {The Analysis of Linear Partial Differential Operators {III}},
  publisher = {Springer},
  address   = {Berlin},
  year      = {1985}
}

@book{nocedal2006numerical,
  title={Numerical Optimization},
  author={Nocedal, Jorge and Wright, Stephen J},
  year={2006},
  publisher={Springer}
}

@book{absil2008optimization,
  title={Optimization Algorithms on Matrix Manifolds},
  author={Absil, P.-A. and Mahony, Robert and Sepulchre, Rodolphe},
  year={2008},
  publisher={Princeton University Press}
}

@book{rockafellar2009variationalanalysis,
  title={Variational Analysis},
  author={Rockafellar, R Tyrrell and Wets, Roger J-B},
  volume={317},
  year={2009},
  publisher={Springer Science \& Business Media}
}

@book{lee2012manifolds,
  title={Smooth Manifolds},
  author={Lee, John M},
  year={2012},
  publisher={Springer}
}

@book{boumal2023introduction,
  title={An Introduction to Optimization on Smooth Manifolds},
  author={Boumal, Nicolas},
  year={2023},
  publisher={Cambridge University Press}
}

@article{cason2013iterative,
  title={Iterative methods for low rank approximation of graph similarity matrices},
  author={Cason, Thomas P and Absil, P.-A. and Van Dooren, Paul},
  journal={Linear Algebra and its Applications},
  volume={438},
  number={4},
  pages={1863--1882},
  year={2013},
  publisher={Elsevier}
}

@article{li2020jotaspectral,
  title={Matrix optimization over low-rank spectral sets: stationary points and local and global minimizers},
  author={Li, Xinrong and Xiu, Naihua and Zhou, Shenglong},
  journal={Journal of Optimization Theory and Applications},
  volume={184},
  pages={895--930},
  year={2020},
  publisher={Springer}
}

@article{li2023normalboundedaffine,
  title={Normal cones intersection rule and optimality analysis for low-rank matrix optimization with affine manifolds},
  author={Li, Xinrong and Luo, Ziyan},
  journal={SIAM Journal on Optimization},
  volume={33},
  number={3},
  pages={1333--1360},
  year={2023},
  publisher={SIAM}
}

@article{levin2025effect,
  title={The effect of smooth parametrizations on nonconvex optimization landscapes},
  author={Levin, Eitan and Kileel, Joe and Boumal, Nicolas},
  journal={Mathematical Programming},
  volume={209},
  number={1},
  pages={63--111},
  year={2025},
  publisher={Springer}
}

@article{yang2025spacedecouple,
  title={A space-decoupling framework for optimization on bounded-rank matrices with orthogonally invariant constraints},
  author={Yang, Yan and Gao, Bin and Yuan, Ya-xiang},
  journal={Mathematical Programming},
  pages={1--53},
  year={2026},
  publisher={Springer}
}

@article{yang20252ndVariety,
  title={Variational analysis of determinantal varieties},
  author={Yang, Yan and Gao, Bin and Yuan, Ya-xiang},
  journal={arXiv preprint arXiv:2511.22613},
  year={2025}
}

@article{nickel2017poincare,
  title={Poincar{\'e} embeddings for learning hierarchical representations},
  author={Nickel, Maximillian and Kiela, Douwe},
  journal={Advances in neural information processing systems},
  volume={30},
  year={2017}
}

@inproceedings{jawanpuria2019lowrankhyperbolic,
  title={Low-rank approximations of hyperbolic embeddings},
  author={Jawanpuria, Pratik and Meghwanshi, Mayank and Mishra, Bamdev},
  booktitle={2019 IEEE 58th conference on decision and control},
  pages={7159--7164},
  year={2019},
  organization={IEEE}
}

@article{peng2025normalized,
  title={Normalized tensor train decomposition},
  author={Peng, Renfeng and Zhu, Chengkai and Gao, Bin and Wang, Xin and Yuan, Ya-xiang},
  journal={arXiv preprint arXiv:2511.04369},
  year={2025}
}

@article{lewis2008alternatingmanifolds,
  title={Alternating projections on manifolds},
  author={Lewis, Adrian S and Malick, J{\'e}r{\^o}me},
  journal={Mathematics of Operations Research},
  volume={33},
  number={1},
  pages={216--234},
  year={2008},
  publisher={INFORMS}
}

@article{lewis2009localaveraged,
  title={Local linear convergence for alternating and averaged nonconvex projections},
  author={Lewis, Adrian S and Luke, D Russell and Malick, J{\'e}r{\^o}me},
  journal={Foundations of Computational Mathematics},
  volume={9},
  number={4},
  pages={485--513},
  year={2009},
  publisher={Springer}
}

@article{andersson2013alternating_cleanintersection,
  title={Alternating projections on nontangential manifolds},
  author={Andersson, Fredrik and Carlsson, Marcus},
  journal={Constructive Approximation},
  volume={38},
  number={3},
  pages={489--525},
  year={2013},
  publisher={Springer}
}

@article{bauschke2013restrictedtheory,
  title={Restricted normal cones and the method of alternating projections: theory},
  author={Bauschke, Heinz H and Luke, D Russell and Phan, Hung M and Wang, Xianfu},
  journal={Set-Valued and Variational Analysis},
  volume={21},
  number={3},
  pages={431--473},
  year={2013},
  publisher={Springer}
}

@article{bauschke2013restrictedapplication,
  title={Restricted normal cones and the method of alternating projections: applications},
  author={Bauschke, Heinz H and Luke, D Russell and Phan, Hung M and Wang, Xianfu},
  journal={Set-Valued and Variational Analysis},
  volume={21},
  number={3},
  pages={475--501},
  year={2013},
  publisher={Springer}
}

@article{drusvyatskiy2015intrinsictransversality,
  title={Transversality and alternating projections for nonconvex sets},
  author={Drusvyatskiy, Dmitriy and Ioffe, Alexander D and Lewis, Adrian S},
  journal={Foundations of Computational Mathematics},
  volume={15},
  number={6},
  pages={1637--1651},
  year={2015},
  publisher={Springer}
}

@article{noll2016separableMAP,
  title={On local convergence of the method of alternating projections},
  author={Noll, Dominikus and Rondepierre, Aude},
  journal={Foundations of Computational Mathematics},
  volume={16},
  number={2},
  pages={425--455},
  year={2016},
  publisher={Springer}
}

@article{ioffe2017transversality,
  title={Transversality in variational analysis},
  author={Ioffe, Alexander D},
  journal={Journal of Optimization Theory and Applications},
  volume={174},
  number={2},
  pages={343--366},
  year={2017},
  publisher={Springer}
}

@article{drusvyatskiy2019inexactapproximate,
  title={Local linear convergence for inexact alternating projections on nonconvex sets},
  author={Drusvyatskiy, Dmitriy and Lewis, Adrian S},
  journal={Vietnam Journal of Mathematics},
  volume={47},
  number={3},
  pages={669--681},
  year={2019},
  publisher={Springer}
}

@article{budzinskiy2025quasioptimal,
  title={Quasioptimal alternating projections and their use in low-rank approximation of matrices and tensors},
  author={Budzinskiy, Stanislav},
  journal={Numerische Mathematik},
  volume={157},
  number={5},
  pages={1491--1535},
  year={2025},
  publisher={Springer}
}

@article{rosen1960gradientI,
  title={The gradient projection method for nonlinear programming. {P}art {I}. {L}inear constraints},
  author={Rosen, Jo Bo},
  journal={Journal of the Society for Industrial and Applied Mathematics},
  volume={8},
  number={1},
  pages={181--217},
  year={1960},
  publisher={SIAM}
}

@article{rosen1961gradientII,
  title={The gradient projection method for nonlinear programming. {P}art {II}. {N}onlinear constraints},
  author={Rosen, J Ben},
  journal={Journal of the Society for Industrial and Applied Mathematics},
  volume={9},
  number={4},
  pages={514--532},
  year={1961},
  publisher={SIAM}
}

@article{frost1972CGP,
  title={An algorithm for linearly constrained adaptive array processing},
  author={Frost, Otis Lamont},
  journal={Proceedings of the IEEE},
  volume={60},
  number={8},
  year={1972}
}

@article{nocedal1985projectedHessian,
  title={Projected {H}essian updating algorithms for nonlinearly constrained optimization},
  author={Nocedal, Jorge and Overton, Michael L},
  journal={SIAM Journal on Numerical Analysis},
  volume={22},
  number={5},
  pages={821--850},
  year={1985},
  publisher={SIAM}
}

@article{yuan2001nullspace,
  title={A null space algorithm for constrained optimization},
  author={Yuan, Ya-xiang},
  journal={Advances in Scientific Computing, Science Press, Beijing},
  pages={210--218},
  year={2001}
}

@inproceedings{ablin2022fastlanding,
  title={Fast and accurate optimization on the orthogonal manifold without retraction},
  author={Ablin, Pierre and Peyr{\'e}, Gabriel},
  booktitle={International Conference on Artificial Intelligence and Statistics},
  pages={5636--5657},
  year={2022},
  organization={PMLR}
}

@article{gao2022landingStiefel,
  title={Optimization flows landing on the {S}tiefel manifold},
  author={Gao, Bin and Vary, Simon and Ablin, Pierre and Absil, P.-A.},
  journal={IFAC-PapersOnLine},
  volume={55},
  number={30},
  pages={25--30},
  year={2022},
  publisher={Elsevier}
}

@inproceedings{schechtman2023ODCGM,
  title={Orthogonal directions constrained gradient method: from non-linear equality constraints to {S}tiefel manifold},
  author={Schechtman, Sholom and Tiapkin, Daniil and Muehlebach, Michael and Moulines, Eric},
  booktitle={The Thirty Sixth Annual Conference on Learning Theory},
  pages={1228--1258},
  year={2023},
  organization={PMLR}
}

@article{ablin2024infeasible,
  title={Infeasible deterministic, stochastic, and variance-reduction algorithms for optimization under orthogonality constraints},
  author={Ablin, Pierre and Vary, Simon and Gao, Bin and Absil, P.-A.},
  journal={Journal of Machine Learning Research},
  volume={25},
  number={389},
  pages={1--38},
  year={2024}
}

@InProceedings{vary24landinggeneralStiefel,
  title = 	 {Optimization without retraction on the random generalized {S}tiefel manifold},
  author =       {Vary, Simon and Ablin, Pierre and Gao, Bin and Absil, P.-A.},
  booktitle = 	 {Proceedings of the 41st International Conference on Machine Learning},
  pages = 	 {49226--49248},
  year = 	 {2024},
  volume = 	 {235},
  publisher =    {PMLR},
}

@inproceedings{song2025distributedlanding,
title={Distributed retraction-free and communication-efficient optimization on the {S}tiefel manifold},
author={Yilong Song and Peijin Li and Bin Gao and Kun Yuan},
booktitle={International Conference on Machine Learning},
year={2025}
}

@article{goyens2026riemannian,
      title={The {R}iemannian Landing Method: from projected gradient flows to {SQP}}, 
      author={Florentin Goyens and Florian Feppon},
      year={2026},
      journal={arXiv preprint arXiv:2603.24309},
}

@article{rockafellar1976ALMconvex,
  title={Augmented {L}agrangians and applications of the proximal point algorithm in convex programming},
  author={Rockafellar, R Tyrrell},
  journal={Mathematics of operations research},
  volume={1},
  number={2},
  pages={97--116},
  year={1976},
  publisher={INFORMS}
}

@article{andreani2008ALM_lowerlevel,
  title={On augmented {L}agrangian methods with general lower-level constraints},
  author={Andreani, Roberto and Birgin, Ernesto G and Mart{\'\i}nez, Jos{\'e} Mario and Schuverdt, Mar{\'\i}a Laura},
  journal={SIAM Journal on Optimization},
  volume={18},
  number={4},
  pages={1286--1309},
  year={2008},
  publisher={SIAM}
}

@article{xiao2025exactconvex,
  title={An exact penalty approach for equality constrained optimization over a convex set},
  author={Xiao, Nachuan and Tang, Tianyun and Wang, Shiwei and Toh, Kim-Chuan},
  journal={arXiv preprint arXiv:2505.02495},
  year={2025}
}

@article{yang2014optimalityRie_nonlinear,
  title={Optimality conditions for the nonlinear programming problems on {R}iemannian manifolds},
  author={Yang, Wei Hong and Zhang, Lei-Hong and Song, Ruyi},
  journal={Pacific Journal of Optimization},
  volume={10},
  number={2},
  pages={415--434},
  year={2014}
}

@article{bergmann2019intrinsicKKT,
  title={Intrinsic formulation of {KKT} conditions and constraint qualifications on smooth manifolds},
  author={Bergmann, Ronny and Herzog, Roland},
  journal={SIAM Journal on Optimization},
  volume={29},
  number={4},
  pages={2423--2444},
  year={2019},
  publisher={SIAM}
}

@article{liu2020simple,
  title={Simple algorithms for optimization on {R}iemannian manifolds with constraints},
  author={Liu, Changshuo and Boumal, Nicolas},
  journal={Applied Mathematics \& Optimization},
  volume={82},
  number={3},
  pages={949--981},
  year={2020},
  publisher={Springer}
}

@article{schiela2021Riesqp,
  title={An {SQP} method for equality constrained optimization on {H}ilbert manifolds},
  author={Schiela, Anton and Ortiz, Julian},
  journal={SIAM Journal on Optimization},
  volume={31},
  number={3},
  pages={2255--2284},
  year={2021},
  publisher={SIAM}
}

@article{obara2022RieSQP,
  title={Sequential quadratic optimization for nonlinear optimization problems on {R}iemannian manifolds},
  author={Obara, Mitsuaki and Okuno, Takayuki and Takeda, Akiko},
  journal={SIAM Journal on Optimization},
  volume={32},
  number={2},
  pages={822--853},
  year={2022},
  publisher={SIAM}
}

@article{zhou2023semismoothNewton,
  title={A semismooth {N}ewton based augmented {L}agrangian method for nonsmooth optimization on matrix manifolds},
  author={Zhou, Yuhao and Bao, Chenglong and Ding, Chao and Zhu, Jun},
  journal={Mathematical Programming},
  volume={201},
  number={1},
  pages={1--61},
  year={2023},
  publisher={Springer}
}

@article{jia2023augmented,
  title={An augmented {L}agrangian method for optimization problems with structured geometric constraints},
  author={Jia, Xiaoxi and Kanzow, Christian and Mehlitz, Patrick and Wachsmuth, Gerd},
  journal={Mathematical Programming},
  volume={199},
  number={1},
  pages={1365--1415},
  year={2023},
  publisher={Springer}
}

@article{lai2024riemannianint,
  title={{R}iemannian interior point methods for constrained optimization on manifolds},
  author={Lai, Zhijian and Yoshise, Akiko},
  journal={Journal of Optimization Theory and Applications},
  volume={201},
  number={1},
  pages={433--469},
  year={2024},
  publisher={Springer}
}

@article{andreani2024CQ_RieALM,
  title={Constraint qualifications and strong global convergence properties of an augmented {L}agrangian method on {R}iemannian manifolds},
  author={Andreani, Roberto and Couto, Kelvin R and Ferreira, Orizon P and Haeser, Gabriel},
  journal={SIAM Journal on Optimization},
  volume={34},
  number={2},
  pages={1799--1825},
  year={2024},
  publisher={SIAM}
}

@article{andreani2026globalRALM,
  title={Global convergence of an augmented {L}agrangian method for nonlinear programming via {R}iemannian optimization},
  author={Andreani, Roberto and Couto, Kelvin R and Ferreira, Orizon P and Haeser, Gabriel and Prudente, Leandro F},
  journal={SIAM Journal on Optimization},
  volume={36},
  number={1},
  pages={466--501},
  year={2026},
  publisher={SIAM}
}

@article{beck2016minimizationC_B,
  title={On the minimization over sparse symmetric sets: projections, optimality conditions, and algorithms},
  author={Beck, Amir and Hallak, Nadav},
  journal={Mathematics of Operations Research},
  volume={41},
  number={1},
  pages={196--223},
  year={2016},
  publisher={INFORMS}
}

@unpublished{chenhuang2026sparse_stiefel,
  author = {Chen, Shixiang and Huang, Wen},
  title  = {Manifold identification and second-order algorithms for $\ell_1$-regularization on the {S}tiefel manifold},
  note   = {Talk at ICCOPT 2025, Los Angeles, CA. \url{https://iccopt2025usc.sched.com/event/21ZZR}},
  year   = {2025},
}

@article{chen2020ManPG,
  title={Proximal gradient method for nonsmooth optimization over the {S}tiefel manifold},
  author={Chen, Shixiang and Ma, Shiqian and Man-Cho So, Anthony and Zhang, Tong},
  journal={SIAM Journal on Optimization},
  volume={30},
  number={1},
  pages={210--239},
  year={2020}
}

@article{ozolins2013compressed,
  title={Compressed modes for variational problems in mathematics and physics},
  author={Ozoli{\c{n}}{\v{s}}, Vidvuds and Lai, Rongjie and Caflisch, Russel and Osher, Stanley},
  journal={Proceedings of the National Academy of Sciences},
  volume={110},
  number={46},
  pages={18368--18373},
  year={2013}
}

@article{xiong2026langding2,
  title={A second-order method landing on the {S}tiefel manifold via {N}ewton--{S}chulz iteration},
  author={Xinhui Xiong and Bin Gao and P.-A. Absil},
  journal={arXiv preprint arXiv:2605.02838},
  year={2026}
}

@article{chu2005lowrankoblique,
  title={On the low-rank approximation of data on the unit sphere},
  author={Chu, M and Del Buono, Nicoletta and Lopez, Luciano and Politi, Tiziano},
  journal={SIAM Journal on Matrix Analysis and Applications},
  volume={27},
  number={1},
  pages={46--60},
  year={2005},
  publisher={SIAM}
}

@article{liu2024SLPG,
  title={A penalty-free infeasible approach for a class of nonsmooth optimization problems over the {S}tiefel manifold},
  author={Liu, Xin and Xiao, Nachuan and Yuan, Ya-xiang},
  journal={Journal of Scientific Computing},
  volume={99},
  number={2},
  pages={30},
  year={2024},
  publisher={Springer}
}

@article{miller1995wordnet,
  title={Word{N}et: a lexical database for {E}nglish},
  author={Miller, George A},
  journal={Communications of the ACM},
  volume={38},
  number={11},
  pages={39--41},
  year={1995},
  publisher={ACM New York, NY, USA}
}

@article{olikier2026tangentconeproof,
  title={The tangent cone to the real determinantal variety: various expressions and a proof},
  author={Olikier, Guillaume and Mlinari{\'c}, Petar and Absil, P-A and Uschmajew, Andr{\'e}},
  journal={Set-Valued and Variational Analysis},
  volume={34},
  number={2},
  pages={8},
  year={2026},
  publisher={Springer}
}

@article{vandereycken2013lowrankcompletion,
  title={Low-rank matrix completion by {R}iemannian optimization},
  author={Vandereycken, Bart},
  journal={SIAM Journal on Optimization},
  volume={23},
  number={2},
  pages={1214--1236},
  year={2013},
  publisher={SIAM}
}

@article{xiao2025quadraticMAP,
  title={A Quadratically Convergent Alternating Projection Method for Nonconvex Sets},
  author={Xiao, Nachuan and Wang, Shiwei and Tang, Tianyun and Toh, Kim-Chuan},
  journal={arXiv preprint arXiv:2511.22916},
  year={2025}
}

@article{chen2026retractionsalternatingprojections,
      title={Retractions by Alternating Projections}, 
      author={Shixiang Chen and Yixiao He and Wen Huang},
      year={2026},
      journal={arXiv preprint arXiv:2605.17384},

}

@article{si2026unifiedlanding,
  title={A Unified Landing Framework for Equality-Constrained Optimization},
  author={Si, Wutao and Malick, Jerome},
  year={2026},
  url = {https://hal.science/hal-05487561},
}

\end{document}